\documentclass[11pt]{amsart}
\usepackage{geometry}                
\usepackage{graphicx}
\usepackage{subfig}
\usepackage{psfrag}
\usepackage{labelfig}
\usepackage{amssymb}
\usepackage{epstopdf}
\usepackage{color}
\newtheorem{theorem}{Theorem}[section]    
\newtheorem{lemma}[theorem]{Lemma}          
\newtheorem{proposition}[theorem]{Proposition}  

\newtheorem{corollary}[theorem]{Corollary} 

\theoremstyle{definition}
\newtheorem{definition}[theorem]{Definition}

\newtheorem{example}[theorem]{Example}    
\newtheorem*{remark}{Remark}             
\numberwithin{equation}{section}

\newcommand{\e}{\varepsilon}

\newcommand{\F}{\mathcal F_{ob} }
\newcommand{\Int}{\textrm{Int}}

\newcommand{\R}{\mathbb{R}}

\newcommand{\D}{\mathcal{D}}

\newcommand{\id}{{\rm{id}}}
\newcommand{\mF}{\mathcal{F} }

\newcommand{\sgn}{{\tt sgn}}
\def\ul{\underline}

\title{Operations on open book foliations} 

\author{Tetsuya Ito}
\address{Research Institute for Mathematical Sciences, Kyoto university
Kyoto, 606-8502, Japan}
\email{tetitoh@kurims.kyoto-u.ac.jp}
\urladdr{http://kurims.kyoto-u.ac.jp/~tetitoh/}

\author{Keiko Kawamuro}
\address{Department of Mathematics \\ 
The University of Iowa \\ Iowa City, IA 52240, USA}
\email{kawamuro@iowa.uiowa.edu}
\date{\today}

\begin{document}
\maketitle
\begin{abstract}
We study b-arc foliation change and exchange move of open book foliations 
which generalize the corresponding operations in braid foliation theory. We also define a bypass move as an analogue of Honda's bypass attachment operation. 

As applications, we study how open book foliations change under a stabilization of the open book. 
We also generalize Birman-Menasco's split/composite braid theorem: 
We show that closed braid representatives of a split (resp. composite) link in a certain open book can be converted to a split (resp. composite) closed braid by applying exchange moves finitely many times. 
\end{abstract}


\section{Introduction}

This is a sequel of the papers \cite{ik1-1,ik1-2,ik2} on open book foliations in which techniques to study the topology and contact structures of 3-manifolds are developed. 
The idea of an open book foliation originally came from the works of Bennequin \cite{Ben} and Birman and Manasco \cite{BM4, BM2, BM5, BM1, BM6, BM3, bm1, bm2}.
 

In this paper we study three types of operations on open book foliations on surfaces that are realized by isotopies of the surfaces; {\em b-arc foliation change} (\S \ref{sec:foliation-change}), {\em bypass move} (\S \ref{sec: Bypass move}) and {\em exchange move} (\S \ref{sec:exchangemove}). 

A b-arc foliation change and an exchange move are generalizations of Birman-Menasco's {\em foliation change} and exchange move in braid foliation theory. 
A bypass move can be seen as an analogue of Honda's {\em bypass attachment} in convex surface theory.

It is natural to expect that our b-arc foliation change and exchange move on open book foliations are more complex than Birman and Menasco's original moves on braid foliations. 
In fact, we need additional assumptions that make these operations actually work. 

Roughly speaking, a b-arc foliation change and a bypass move are associated to isotopies interchanging the `heights' of a pair of adjacent saddle points of a surface. A b-arc foliation change treats the case that two saddles have the same sign whereas a bypass move treats the case with opposite signs. 

These isotopies are local in the sense that they take   place in 3-balls. 
Hence both a b-arc foliation change and a bypass move are local operations on open book foliations. 
Under these operations, the total number of singularities of an open book foliation stays the same.
Moreover, if there are braids passing through the 3-balls, the isotopies preserve the braid isotopy classes.


On the contrary, an isotopy realizing an exchange move  may change the braid isotopy class. (The braid index and the transverse link type of the braid are preserved.) 
Also the number of singularities of an open book foliation decreases by an exchange move.

In the second half of the paper we discuss two applications: 

We study effect of (de)stablizations of open books on open book foliations in \S\ref{sec:stabilization}. 
We show that the open book foliation of a surface changes in two ways after a stabilization of the open book. 
Next, we see that the resulting two open book foliations are related to each other by bypass moves and exchange moves.

As applications of b-arc foliation change and exchange move operations, in \S\ref{sec:braid theorem} we consider the split/composite closed braid theorems of Birman and Menasco \cite{BM4} in the setting of general open books and prove them under certain conditions.

\section{Preliminaries} 

We assume that the readers are familiar with the basic definitions and properties of open book foliations which can be found in \cite{ik1-1, ik2}. 

Let $(S,\phi)$ be an open book decomposition of a closed oriented $3$-manifold $M$, where $S=S_{g, r}$ is a genus $g$ surface with $r$ boundary components, and $\phi \in {\rm Diff}^+(S, \partial S)$ an orientation preserving differomorphism of $S$ fixing the boundary pointwise.   
The manifold $M$ is often denoted by $M_{(S, \phi)}$. 
Let $B$ denote the {\em binding} of the open book and $\pi:M \setminus B \rightarrow S^{1}$ the fibration whose fiber $S_{t} := \pi^{-1}(t)$ is a {\em page}.

An oriented link $L$ in $M_{(S, \phi)}$ is called a {\em closed braid} with respect to the open book $(S,\phi)$ if $L$ is disjoint from the binding $B$ and positively transverse to each page $S_{t}$. 

Let $F \subset M_{(S, \phi)}$ be an embedded, oriented surface possibly with boundary. 
If $F$ has boundary, $\partial F$, we require that $\partial F$ is a closed braid with respect to $(S, \phi)$. 
Up to perturbation of $F$ the singular foliation $$\F(F)=\left\{ F \cap S_t \ | \ t \in [0, 1] \right\}$$ admits the following conditions (see Theorem 2.5 of \cite{ik1-1}). 

\begin{description}
\item[($\mF$ i)] 
The binding $B$ pierces the surface $F$ transversely in finitely many points. 
Moreover, $p \in B \cap F$ if and only if there exists a disc neighborhood $N_{p} \subset \Int(F)$ of $p$ on which the foliation $\F(N_p)$ is radial with the node $p$, see Figure~\ref{fig:sign}-(1, 2). 
We call $p$ an {\em elliptic} point. 

\item[($\mF$ ii)] 
The leaves of $\F(F)$ along $\partial F$ are transverse to $\partial F$. 

\item[($\mF$ iii)] 
All but finitely many fibers $S_{t}$  intersect $F$ transversely.
Each exceptional fiber is tangent to $F$ at a single point $\in\Int(F)$.
In particular, $\F(F)$ has no saddle-saddle connections.

\item[($\mF$ iv)] 
All the tangencies of $F$ and fibers are of saddle type, see Figure~\ref{fig:sign}-(3, 4). 
We call them {\em hyperbolic} points.

\end{description}

\begin{definition}
We call each connected component of $F \cap S_t$ a {\em leaf}.  
We say a leaf $l$ of $\F(F)$ is {\it regular} if $l$ does not contain a tangency point and is {\it singular} otherwise.
The regular leaves are classified into the following three types:
\begin{enumerate}
\item[a-arc]: An arc where one of its endpoints lies on $B$ and the other lies on $\partial F$.
\item[b-arc]: An arc whose endpoints both lie on $B$.
\item[c-circle]: A simple closed curve.
\end{enumerate} 
\end{definition}

In order to study topology and geometry of 3-manifolds $M_{(S, \phi)}$ it is often important to take the following homotopical properties of leaves into account.

\begin{definition}\label{def of essential}\cite{ik2}
We say that a b-arc $b \subset S_{t}$ is {\em essential} (resp. {\em strongly essential}) if $b$ is not boundary-parallel in $S_{t}\setminus(S_{t} \cap \partial F)$ (resp. $S_t$). 
An elliptic point $v$ is called {\em strongly essential} if every $b$-arc that ends at $v$ is strongly essential.
An open book foliation $\F(F)$ is called ({\em strongly}) {\em essential} if all the $b$-arcs are (strongly) essential. 
\end{definition}

For a b-arc the conditions `{\em boundary parallel in $S_t$}' and `{\em non}-strongly essential' are equivalent.  In this paper we prefer to use the former. 

Essentiality is a natural condition in the sense that if $F$ is incompressible then applying an isotopy that fixes $\partial F$ (if it exists) $F$ admits an essential open book foliation \cite{ik2}.

\begin{definition}\label{def:separating} 
We say a b-arc $b$ in the page $S_t$ is {\em separating} if $b$ separates the page $S_t$ into two regions. 
\end{definition} 

Clearly an inessential or boundary-parallel b-arc is separating. 
We will use this separating condition in Proposition~\ref{prop:sufficient-conditions}, Lemmas \ref{lemma:degeneratebc} and \ref {lemma:sign} below. 




We say that an elliptic point $p$ is {\em positive} (resp. {\em negative}) if the binding $B$ is positively (resp. negatively) transverse to $F$ at $p$.
The sign of the hyperbolic point $q$ is {\em positive} (resp. {\em negative}) if the positive normal direction of $F$ at $q$ agrees (resp. disagrees) with the direction of $t$.
We denote the sign of a singular point $v$ by $\sgn(v)$.
See Figure \ref{fig:sign}.
%
\begin{figure}[htbp]
\begin{center}
\SetLabels
(.3*.92) (1)\\
(.83*.92) (2)\\
(.3*.4) (3)\\
(.83*.4) (4)\\
\endSetLabels
\strut\AffixLabels{\includegraphics[width=130mm]{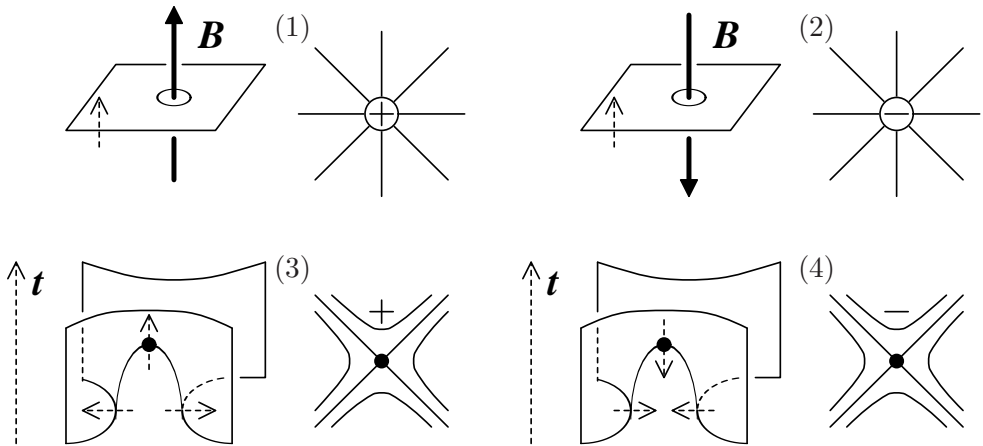}}
\end{center}
\caption{
(1, 2) elliptic points. (3, 4) hyperbolic points. 
Dashed arrows indicate normal vectors to $F$.}\label{fig:sign}
\end{figure}



Hyperbolic singularities in $\F(F)$ are classified into  
six types, according to the types of nearby regular leaves: Type $aa$, $ab$, $bb$, $ac$, $bc$, and $cc$ as depicted in Figure ~\ref{region}.
\begin{figure}[htbp]
\begin{center}
\SetLabels
(0.15*0.55)  $aa$-tile\\
(0.5*0.55)    $ab$-tile\\
(0.84*0.55)  $bb$-tile\\
(0.15*0.04)  $ac$-annulus\\
(0.5*0.04)    $bc$-annulus\\
(0.84*0.04)  $cc$-pants\\
\endSetLabels
\strut\AffixLabels{\includegraphics*[scale=0.5, width=90mm]{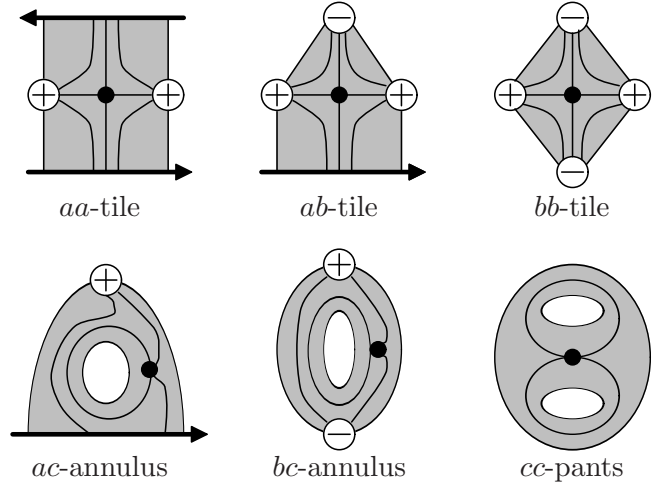}}
\caption{Six types of regions.}\label{region}
\end{center}
\end{figure}
Such a model neighborhood is called a {\em region}.  
We denote by $\sgn(R)$ the sign of the hyperbolic point contained in the region $R$. 

\section{B-arc foliation change}\label{sec:foliation-change}

In this section we generalize Birman-Menasco's {\em foliation change} of braid foliations \cite[p.123]{BM6} to  a {\em b-arc foliation change} of open book foliations (Theorem~\ref{thm:folchange}).

Here is the set up for a $b$-arc foliation change: 
Let $(S,\phi)$ be an open book decomposition of a $3$-manifold $M$ and $\F(F)$ the open book foliation on $F$, where $F$ is a closed surface in the complement of a closed braid $L$ or a Seifert surface of a closed braid $L$.

We will use the underlined letter ``${\underline a}$'' to indicate the image of an arc $a \subset S_t$ superimposed on $S$ by a natural projection $(p, t) \in S_t \mapsto p \in S$. This allows us to compare leaves in different pages.
We assume that the region decomposition of $F$ contains two tiles $R_{1}, R_{2}$ satisfying the following conditions (i)--(iv). See also  Figure~\ref{fig:folchange}-(a):
\begin{description}
\item[(i)] 
$R_i$ $(i=1,2)$ is either an $ab$-tile or a $bb$-tile.
\item[(ii)] 
$\sgn(R_{1})=\sgn(R_{2})=\e \in \{+1, -1\}.$
\item[(iii)] 
$R_{1}$ and $R_{2}$ are adjacent exactly at one b-arc, $b$.
\end{description}
\begin{figure}[htbp]
 \begin{center}
 \SetLabels
(.01*.73) (a)\\ 
(0.1*.66)    $l_4$\\ 
(0.36*.66)  $l_3$\\ 
(.2*.73) $v$\\
(.2*.25) $A$\\ 
(.44*.63)   $B$\\ 
(-.01*.61)  $C$\\ 
(.44*.5) $l_2$\\
(.32*.3) $l_1$\\
(.1*.3)   $l_6$\\
(0*.47)  $l_5$\\ 
(.2*.5)  $b$\\
(.1*.53) $\e$\\
(.35*.53) $\e$\\
(0.35*0.38)  $R_{1}$\\
(0.13*0.38)  $R_{2}$\\
(.53*.95) (b)\\
(0.65*.94) $l_4$\\ 
(0.91*.94)   $l_3$\\ 
(.75*1) $v$\\
(.75*.53) $A$\\ 
(1*.85)   $B$\\ 
(.55*.85)  $C$\\ 
(0.89*.57)   $l_1$\\
(.55*.75)  $l_5$\\ 
(1*.75) $l_2$\\
(.66*.57) $l_6$\\ 
(.73*.67) $\e$\\
(.81*.85) $\e$\\
(0.9*0.8)  $R_{2}'$\\
(0.67*0.7)  $R_{1}'$\\
(.77*.77) $l$\\
(.53*.05) (c)\\
(0.65*.39) $l_4$\\ 
(0.91*.39)  $l_3$\\ 
(.75*.44) $v$\\
(.75*0) $A$\\ 
(1*.3)   $B$\\ 
(.55*.3)  $C$\\ 
(0.89*.04)   $l_1$\\
(.55*.2)  $l_5$\\ 
(1*.2) $l_2$\\ 
(.66*.04) $l_6$\\ 
(.73*.31) $\e$\\
(.81*.14) $\e$\\
(0.93*0.2)  $R_{2}'$\\
(0.8*0.3)  $R_{1}'$\\
(.77*.23) $l$\\
\endSetLabels
\strut\AffixLabels{\includegraphics*[scale=0.5, width=100mm]{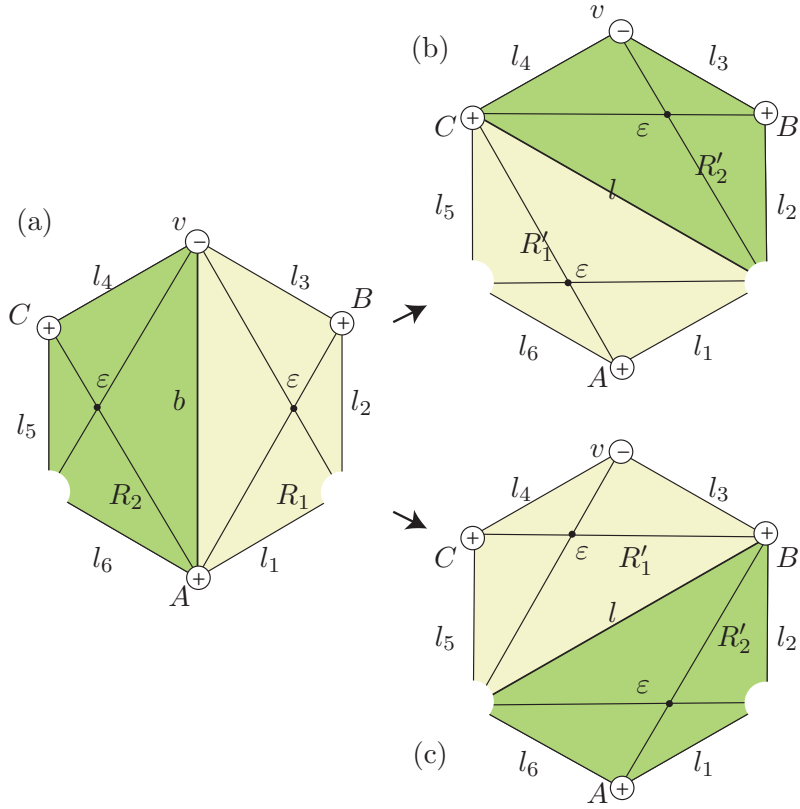}}
\caption{$b$-arc foliation change.}\label{fig:folchange}
\end{center}
\end{figure}
Let $v$ (resp. $A$) be the negative (resp. positive) elliptic point at the end of $b$, and 
$l_1, \cdots, l_6$ be boundary arcs of $R_1 \cup R_2$ as depicted in Figure~\ref{fig:folchange}-(a). 
Suppose that $l_k \subset S_{t_k}$ where $k=1,\cdots,6$ and $t_k \in [0,1)$, and the hyperbolic point of $R_i$ is sitting on the page $S_{\tau_i}$. 
The open book foliation $\F(R_1 \cup R_2)$ imposes the following relations. 
$$
\begin{array}{rcccccl}
&&\tau_1  & < &&& t_2  \\
\max\{t_1, t_3\} &< & \tau_1 &<& \tau_2 & <& \min\{ t_4, t_6\}\\
t_5 &&& < & \tau_2 &&
\end{array} 
$$
In addition to the above conditions (i, ii, iii) we further require that $$\max\{t_1, t_3, t_5\} < \tau_1 < \tau_2  < \min\{t_2, t_4, t_6\}, \mbox{ or}$$ 
\begin{description}
\item[(iv)] 
$t_1= t_3 = t_5 <\tau_1 < \tau_2  <  t_2= t_4= t_6$.
\end{description}

Let $\gamma_i$ denote the describing arc for the hyperbolic point in $R_i$ ($i=1,2$). 
We may assume that $\ul{\gamma_1}$ joins $\ul{l_1}$ and $\ul{l_3}$. See Figure~\ref{fig:tree}. 
\begin{figure}[htbp]
\begin{center}
\SetLabels
(0*.8) $(\e=+1)$\\ 
(.21*.8) $\ul{l_1}$\\
(.1*.5) $\ul{\gamma_1}$\\
(.25*.27) $\ul{\gamma_2}$\\
(0.33*.6)    $\ul{l_5}$\\
(0.24*.96) $A$\\
(0*.25) $B$\\
(.14*.25) $\ul{l_3}$\\
(.25*0) $v$\\
(.4*.27) $C$\\
(.6*.8) $(\e=-1)$\\ 
(.59*.25) $A$\\
(0.84*.96) $B$\\
(0.62*.6)    $\ul{l_1}$\\
(0.77*.29)    $\ul{l_5}$\\
(.72*.6) $\ul{\gamma_1}$\\
(.82*.4) $\ul{\gamma_2}$\\
(.85*.7) $\ul{l_3}$\\
(.85*0) $C$\\
(1.02*.51) $v$\\
\endSetLabels
\strut\AffixLabels{\includegraphics*[scale=0.5, width=120mm]{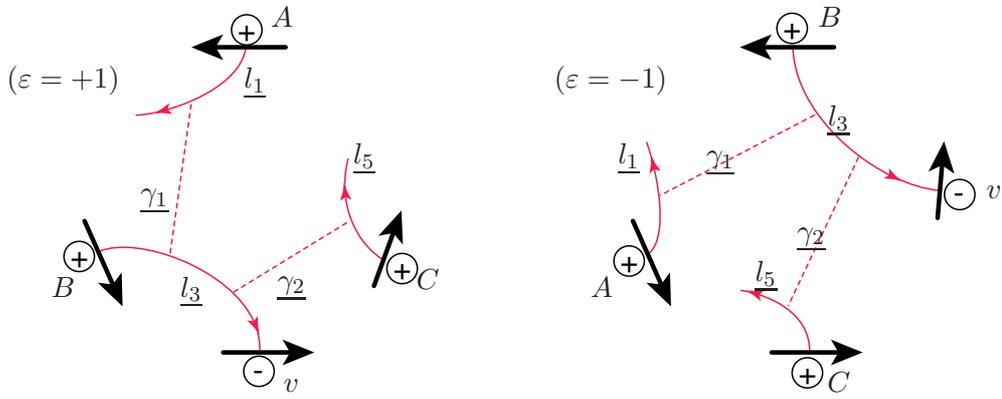}}
\caption{The superimposed graph in $S$.}\label{fig:tree}
\end{center}
\end{figure}
By sliding $\ul{\gamma_2}$ along $\ul b$, we can further  assume that $\ul{\gamma_2}$ joins $\ul{l_3}$ and $\ul{l_5}$. 
Since $\sgn(R_1)=\sgn(R_2) = \e$, if we walk along $\ul{l_3}$ from $B$ to $v$, regardless of the sign $\e$, we meet $\ul{\gamma_1}$ first then $\ul{\gamma_2}$ and both $\ul{\gamma_1}, \ul{\gamma_2}$ lie on the same side of $\ul{l_3}$.
In general the arcs $\ul{l_1}, \ul{l_3}, \ul{b_5}, \ul{\gamma_1}, \ul{\gamma_2}$ may intersect each other.

\begin{theorem}[b-arc foliation change]\label{thm:folchange}
Assume that $R_1, R_2$ satisfy the above conditions {\bf(i)}--{\bf(iv)}. 
Suppose that the graph $\ul{l_1} \cup \ul{l_3} \cup \ul{l_5} \cup \ul{\gamma_1} \cup \ul{\gamma_2}$, see Figure~\ref{fig:tree}, is a 
tree in $S$. 
Then there is an ambient isotopy $\Phi_{\tau}:M \rightarrow M$ supported on $M \setminus B$ such that: 
\begin{enumerate}
\item
$F'=\Phi_{1}(F)$ admits an open book foliation $\F(F')$. If $\F(F)$ is essential, then so is $\F(F')$. 
\item 
The region decomposition of $\F(F')$ contains regions $R'_{1}, R'_{2}$ (see Figure~\ref{fig:folchange}-b,c)
\begin{enumerate}
\item 
of type either $aa$, $ab$, or $bb$-tile,
\item 
$\sgn(R'_{1})=\sgn(R'_{2})=\e$ as in {\rm(ii)},
\item 
$\Phi_1(R_1 \cup R_2)= R_1' \cup R_2'$
\item 
$R'_{1}\cap R'_{2}$ is exactly one leaf $l$ of type $a$ or $b$,
\item
the numbers of the hyperbolic points connected to $v$ and $A$ by a singular leaf decrease both by one, though the total number of hyperbolic points remains the same. 
\end{enumerate}
\item 
$\Phi_t$ preserves the region decomposition of $F\setminus(R_{1} \cup R_{2})$. 
\item 
If $\partial F$ is non-empty $\Phi_{t}(\partial F)$ is a closed braid w.r.t. $(S, \phi)$ for all $t \in [0,1]$, i.e., $L=\partial F$ and $L' = \partial F'$ are braid isotopic. 
\end{enumerate}
\end{theorem}

\begin{proof}
Let $N=N(\ul{l_1} \cup \ul{l_3} \cup \ul{l_5} \cup \ul{\gamma_1} \cup \ul{\gamma_2}) \subset S$ be a regular neighborhood of the graph $G=\ul{l_1} \cup \ul{l_3} \cup \ul{l_5} \cup \ul{\gamma_1} \cup \ul{\gamma_2}$. 
Since $G$ is a tree, $N$ is planar and there is an embedding $\iota: N\hookrightarrow D^{2}$ such that $\iota(\partial S \cap N) \subset \partial D^{2}$. See Figure~\ref{fig:embed}.
\begin{figure}[htbp]
\begin{center}
\SetLabels
(0.18*.55)    $\ul{\gamma_1}$\\
(0.1*.27)    $\ul{b_1}$\\
(0.28*.83)   $\ul{l_1}$\\
(.2*.35) $\ul{l_3}$\\
(0.35*.78)    $\ul{l_5}$\\
(0.28*.43)   $\ul{\gamma_2}$\\
(0.52*.52)  \Large $\iota$\\
\endSetLabels
\strut\AffixLabels{\includegraphics*[scale=0.5, width=100mm]{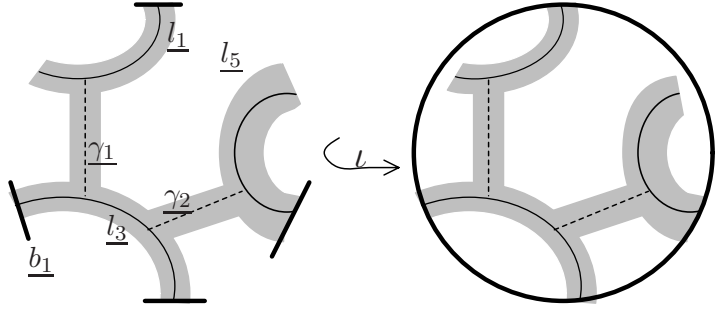}}
\caption{An embedding $\iota: N \hookrightarrow D^2$ when $\e=+1$.}
\label{fig:embed}
\end{center}
\end{figure}
We may assume that the region $R_1 \cup R_2$ is embedded in $N\times [t_1, t_2]$, hence also in $D^2 \times [t_1, t_2]$. 
The foliation on the surface $(\iota \times id)(R_1 \cup R_2) \subset D^2 \times  [t_1, t_2]$ induced by the family of discs $\{D^{2}  \times \{t\} | t \in [t_1, t_2]\}$ is the same as that on $\F(R_1 \cup R_2)$. 
Theorem 2.1 of Birman and Finkelstein \cite{bf} guarantees the existence of a desired isotopy $\Phi_t$.

Here we sketch the transition of $\Phi_t(R_1 \cup R_2)$ from $t=0$ to $t=1$ when $\e=-1$.  
Figure~\ref{fig:folc_1}-(a) depicts the interior of $R_1\cup R_2$, where the two saddles lie on the different pages $S_{\tau_1}$ and $S_{\tau_2}$ of the open book.
\begin{figure}[htbp]
\begin{center}
\SetLabels
(0.08*1)   (a)\\
(0.4*1) (b)\\
(0.72*1)  (c)\\
(-0.02*0.1) $t$\\
(.15*1) $l_1$\\
(.1*.6) $l_5$\\
(.3*1) $l_3$\\
(.24*.07) $l_6$\\
(.28*.27) $l_2$\\
(.2*-.1) $l_4$\\
\endSetLabels
\strut\AffixLabels{\includegraphics[width=130mm]{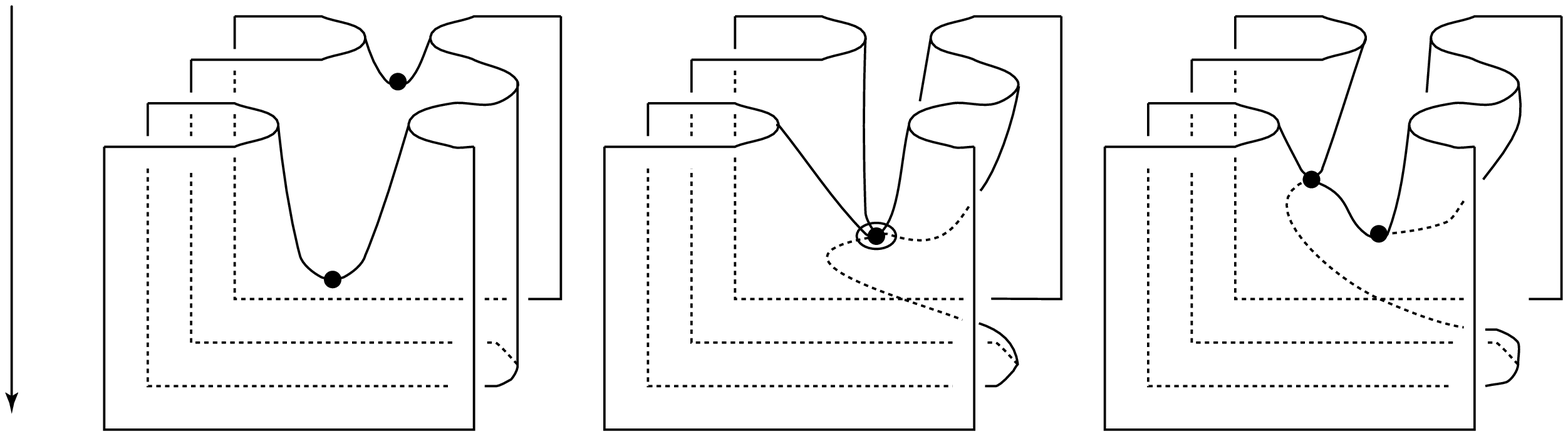}}
\caption{Isotopy $\Phi_t$ of $R_1\cup R_2$ that realizes b-arc foliation change.}
\label{fig:folc_1}
\end{center}
\end{figure}
We perturb the surface so that the saddles get closer until amalgamated to a monkey saddle, or a valence $6$ saddle, see Figure~\ref{fig:folc_1}-(b). 
By further perturbation the singular point splits into two hyperbolic points as shown in Figure~\ref{fig:folc_1}-(c).
The isotopy replaces $\gamma_1, \gamma_2$ (the top row of Figure \ref{fig:folc_0})
\begin{figure}[htbp]
\begin{center}
\SetLabels
(0.32*.6)    $v$\\
(0.62*.6)    $v$\\
(.93*.6) $v$\\
(0.06*.82)   $A$\\
(.69*.8) $A$\\
(0.12*.96) $l_1$\\
(0.27*.96) $l_3$\\
(0.37*.83)   $A$\\
(0.6*.8)   $b$\\
(.5*.42) $l$\\
(.21*.8) $\gamma_1$\\
(.52*.71) $\gamma_2$\\
(.8*.91) $l_2$\\
(.8*.6) $l_4$\\
(.8*.8) $l_6$\\
(.12*.57)   $l_5$\\
(.44*.57)   $l_5$\\
(.21*.25) $\gamma_1'$\\
(.55*.25) $\gamma_2'$\\
\endSetLabels
\strut\AffixLabels{\includegraphics*[scale=0.5, width=120mm]{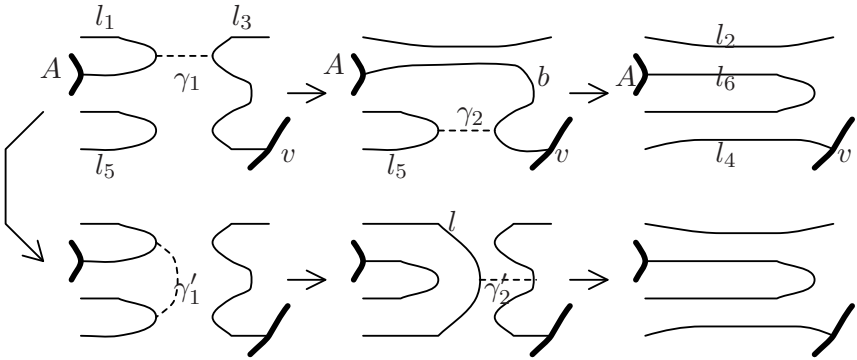}}
\caption{Replacing describing arcs where $\e = -1$.}
\label{fig:folc_0}
\end{center}
\end{figure}
with $\gamma_1', \gamma_2'$ (cf. the bottom row). 
This results in change in the open book foliation of $R_1 \cup R_2$ as depicted in Figure~\ref{fig:folchange}. 
For example, Figure~\ref{fig:folc_0} corresponds to the transition (a) $\rightarrow$ (b) in Figure~\ref{fig:folchange}.

Finally it is easy to see the assersion (1): if $\F(F)$ is essential, then so is $\F(F')$. 
If $\F(F')$ is inessential then the leaf $l=R_1' \cap R_2'$ must be inessential, which implies at least one of the leaves $l_i$ must be inessential. 
(In the case of Figure~\ref{fig:folc_0}, the leaves $l_3$ or $l_6$ is inessential.) 
\end{proof}

In general, checking the assumption of Theorem~\ref{thm:folchange} is not so simple, but there is one sufficient condition which is easier to check: 



\begin{proposition}\label{prop:sufficient-conditions}
In addition to the conditions {\bf(i)}--{\bf(iv)}, assume further that the common b-arc $b$ of the tiles $R_1$ and $R_2$ is separating in the sense of Definition~\ref{def:separating}.
Then the graph $\ul{l_1} \cup \ul{l_3} \cup \ul{l_5} \cup \ul{\gamma_1} \cup \ul{\gamma_2}$ is a 
tree in $S$.
\end{proposition}

\begin{proof}
Suppose that $\e=+1$ (for the case $\e= -1$ a parallel argument holds).  
Let $S \setminus \ul{b} = D \sqcup D'$, where 
$D$ (resp. $D'$) is the region on the left (right) side of $\ul b$ as we walk along $\ul b$ from the positive elliptic point $A$ to the negative elliptic point $v$. 
See Figure~\ref{tree-and-disc}. 
\begin{figure}[htbp]
\begin{center}
\SetLabels
(0*.95) Region $D$\\ 
(.2*.67) $D$\\
(.28*0) $v$\\
(0.27*.96) $A$\\
(0*.25) $B$\\
(.46*.27) $C$\\
(.21*.8) $\ul{l_1}$\\
(0.35*.6)    $\ul{l_5}$\\
(.1*.5) $\ul{\gamma_1}$\\
(.25*.36) $\ul{\gamma_2}$\\
(.14*.25) $\ul{l_3}$\\
%
(.6*.95) Region $D'$\\ 
(.65*.27) $D'$\\
(.84*0) $v$\\
(0.82*.96) $A$\\
(.59*.25) $B$\\
(1.02*.27) $C$\\
(0.72*.3)  $\ul{l_3}$\\
(.72*.6) $\ul{\gamma_1}$\\
(.82*.4) $\ul{\gamma_2}$\\
(.88*.51) $\ul{l_5}$\\
(.77*.8) $\ul{l_1}$\\
\endSetLabels
\strut\AffixLabels{\includegraphics*[scale=0.5, width=110mm]{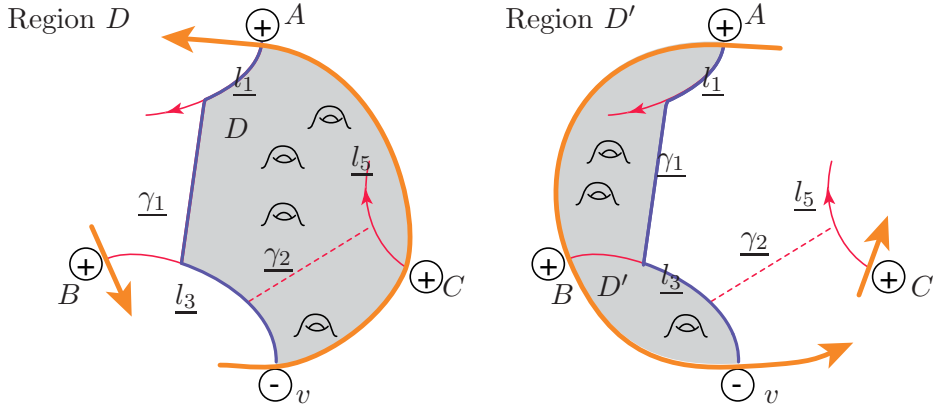}}
\caption{The regions $D$ and $D'$ when $\e=+1$. Vertices $B$ and $C$ may not be on the same binding component where $A$ and $v$ lie.}\label{tree-and-disc}
\end{center}
\end{figure}
Note that $v$ and $A$ lie on the same boundary component of $S$ because $b$ is separating. 
The vertex $C$ (resp. $B$) lies on $\partial D$ (resp. $\partial D'$) but not necessarily on the same boundary component on which $v$ and $A$ lie. 
Since $l_1, \gamma_1$ and $l_3$ are contained in the same page $S_{t_1}$, their images $\ul{l_{1}}, \ul{l_3},  \ul{\gamma_{1}}$ form a tree in $S$, and the tree $\ul{l_{1}} \cup \ul{l_3} \cup \ul{\gamma_{1}}$ is disjoint from $\Int(D)$. 
The tree $\ul{\gamma_{2}} \cup \ul{l_5}$ is contained in $D$. 
Hence the graph $(\ul{l_{1}} \cup \ul{l_3} \cup \ul{\gamma_{1}}) \cup (\ul{\gamma_{2}} \cup \ul{l_5})$ is a tree in $S$. 
\end{proof}

The above argument implies the following: 

\begin{corollary}\label{claim p-iness}
If $b$ is boundary parallel (so $D$ or $D'$ is a disc region) then $l_3$ or $l_4$ is boundary parallel. 
\end{corollary}

\begin{remark}
The essential point in the above proof is that $\Int(\ul{\gamma_2})$ and 
$\ul{l_1}\cup \ul{l_3}$ 
are disjoint, so our problem is reduced to a problem in braid foliation theory, a theory for the {\em trivial} open book $(D^2, id)$. 
Suppose that $\ul{\gamma_{2}}$ is parallel to $\ul{\gamma_{1}}$ as in Figure \ref{fig:nestedsaddle}. 
\begin{figure}[htbp]
\begin{center}
\SetLabels
(0.08*.48)    $\gamma_{1}$\\
(0.55*.2)   $\gamma_{2}$\\
(.05*.3) $l_3$\\
(.04*.63) $l_1$\\
(.25*.79) $l_5$\\
\endSetLabels
\strut\AffixLabels{\includegraphics*[scale=0.5, width=120mm]{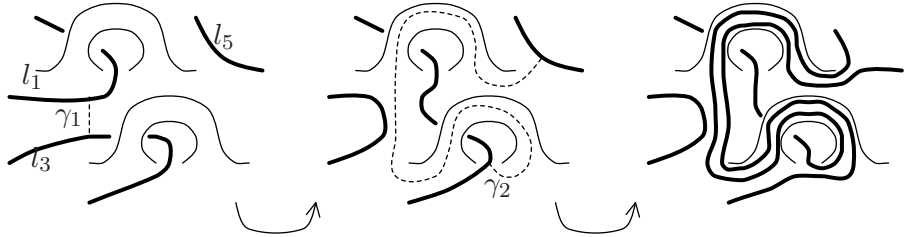}}
\caption{Nested saddles.}
\label{fig:nestedsaddle}
\end{center}
\end{figure}
The right sketch shows the saddles for $\gamma_1$ and $\gamma_2$ are {\em nested}. 
The saddle of $\gamma_{2}$ can exist only after the saddle of $\gamma_{1}$, so the trick of replacing the order of describing arcs (cf. Figure~\ref{fig:folc_0}) does not work. 

The existence of nested saddles is a unique feature of open book foliations.  
In braid foliation theory no b-arcs are strongly essential because the page $S$ is a disc, so by Proposition~\ref{prop:sufficient-conditions} if $\sgn(R_{1})=\sgn(R_{2})$ the graph $\ul{l_1} \cup \ul{l_3} \cup \ul{l_5} \cup \ul{\gamma_1} \cup \ul{\gamma_2}$ is always a tree in $S$ and nested saddles do not exist. 
\end{remark}

\begin{remark}
One might consider an {\em a-arc} foliation change 
under a similar setting where two tiles of the same sign are adjacent along an a-arc, instead of a b-arc. However, ``a-arc foliation change'' does {\em not} work in general. 
This is why we call our operation {\em b-arc} foliation change, rather than simply calling it foliation change. We thank Bill Menasco for pointing this out and informing us importance of the separating condition on the b-arc $b$ in Proposition~\ref{prop:sufficient-conditions}.
\end{remark}

\section{Bypass move}\label{sec: Bypass move}

In the setting of a b-arc foliation change the two adjacent tiles $R_1, R_2$ must have the same sign.
Now a natural question arise: how about the case that two adjacent regions have {\em opposite} signs? 

It has been observed by Birman and Menasco in braid foliation theory that the opposite sign case is more complicated than the same sign case: 
They found that the complement of the hexagon region $F\setminus (R_1\cup R_2)$ or a closed braid may prevent the desired height exchange of the saddles (see \cite[Fig 11b]{BM4}). Thus validity of similar moves in open book foliation theory 
should reflect global feature of the surface $F$.

In this section we study when the `heights' of hyperbolic points of opposite signs are exchangeable.  
A short answer to this question would be ``when there exists a bypass-rectangle'', which we define shortly. 
We start by defining {\em dividing sets} whose idea comes from Giroux's {\em dividing sets} for convex surfaces \cite[\S 2]{Giroux-convex}, see also Honda's \cite[\S 3.1.3]{Honda}. 

\begin{definition}[Dividing set]\label{def:pos-neg-region}
Let $F \subset M_{(S, \phi)}$ be a surface admitting an open book foliation $\F(F)$ with no c-circles. (In \cite{ik1-1}  we prove that by finger moves we can always get rid of c-circles.) 
Let $\Gamma \subset F$ be a set of properly embedded arcs and circles that decompose $F$ into regions $F_+$ and $F_-$ such that 
\begin{itemize}
\item
$F \setminus \Gamma = F_+ \sqcup F_-$.
\item
As sets (forgetting orientations) $\Gamma = \partial F_+ \setminus \partial F = \partial F_- \setminus \partial F$. 
\item
The leaves of $\F(F)$ along $\Gamma$ are oriented out of the region $F_+$ and into $F_-$. 
\item 
$F_+$ contains all the positive singularities of $\F(F)$.
\item 
$F_-$ contains all the negative singularities of $\F(F)$.
\end{itemize} 
We call $\Gamma$ the {\em dividing set} of $\F(F)$. 
\end{definition} 

Given an open book foliation $\F(F)$ with no c-circles, the region $F_-$ can be identified, up to isotopy, with a collar neighborhood of the graph $G_{--}$ of $\F(F)$ (see \cite{ik1-1} for definition), hence $\Gamma$ is uniquely determined up to isotopy.

Next we define a bypass rectangle inspired by Honda's {\em bypass half-disc} \cite[\S 3.4]{Honda}. 

\begin{definition}[Bypass rectangle]

Let $\tilde \D\subset M_{(S, \phi)}$ be a disc admitting an open book foliation such that $\F(\tilde \D)$ is a degenerate aa-tile. 
Hence $\F(\tilde \D)$ contains two positive elliptic points and one hyperbolic point of sign $\varepsilon \in \{\pm1\}$. 
Let $\D \subset \tilde \D \subset M_{(S, \phi)}$ be a rectangle region such that 
\begin{enumerate}
\item $\D$ contains the hyperbolic pint of $\F(\tilde \D)$. 
\item The boundary $\partial \D$ consists of four piecewise smooth curves $\delta_1, \cdots, \delta_4$ such that 
\begin{enumerate}
\item
$\delta_1, \delta_3 \subset \partial \tilde \D \cap \partial \D$,  
\item 
$\delta_2, \delta_4$ are properly embedded arcs in $\tilde \D$. 
\end{enumerate} 
\item The leaves of $\F(\tilde \D)$ transversely intersect each $\delta_i$. The orientation of the leaves is pointing out of (resp. into) $\D$ along $\delta_1, \delta_3$ (resp. $\delta_2, \delta_4$). 
\end{enumerate} 
\begin{figure}[htbp]
\begin{center}
\SetLabels
(.45*0) $\delta_1$\\
(.45*1) $\delta_3$\\
(.27*.4) $\delta_4$\\
(.72*.4) $\delta_2$\\
(.52*.55) $\varepsilon$\\
(.27*.04) $p$\\
(.73*.04) $q$\\
(.27*1) $p'$\\
(.73*1) $q'$\\
\endSetLabels
\strut\AffixLabels{\includegraphics*[width=120mm]{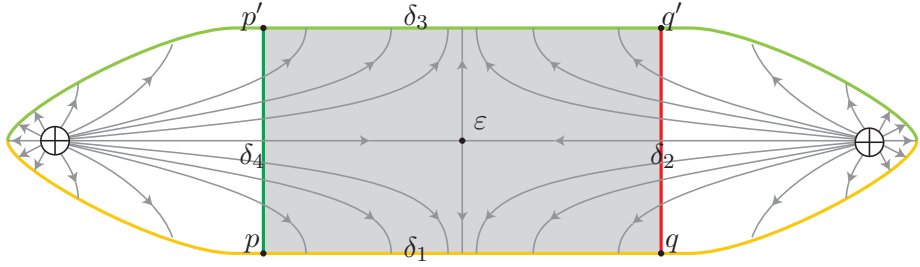}}
\caption{Bypass rectangle $\D$ (shaded) and degenerate aa-tile $\tilde \D$.}\label{fig:bypass}
\end{center}
\end{figure} 
See Figure~\ref{fig:bypass}. 
Denote the four vertices of $\D$ by $p, p', q, q'$. 
We call $\D$ a {\em bypass rectangle} of sign $\e:=\sgn(\D)$. 
\end{definition}

\begin{definition}[{\tt Type1}, {\tt Type2} hexagon $R$]
Let $F \subset M_{(S, \phi)}$ be a surface admitting an open book foliation. 
Suppose that $F$ contains a hexagon region $R$ consisting of two bb-tiles of opposite signs meeting along a b-arc as in Sketch (1) of Figure~\ref{fig:bypass-hexagon}. 
We name the vertices (elliptic points) $A, B, C, D, E, F$ counterclockwise. We may assume that 
$\sgn(A)=\sgn(C)=\sgn(E)=+1$ and $\sgn(B)=\sgn(D)=\sgn(F)=-1$. 
We require that the boundary b-arcs $\overline{AB}, \overline{CD}, \overline{EF}$ lie on the same page of the open book, and likewise $\overline{BC}, \overline{DE}, \overline{FA}$ lie on another same page. 
\begin{figure}[htbp]
\begin{center}
\SetLabels
(0*.73) (1) hexagon $R$\\ 
(.22*.74) $A$\\
(0*.63) $B$\\
(.085*.46) ${\bf p}$\\
(.37*.5) ${\bf q}$\\
(0*.35) $C$\\ 
(.22*.23) $D$\\
(.43*.35) $E$\\
(.43*.63) $F$\\
(.53*.58) {\tiny retrograde}\\
(.53*.4) {\tiny prograde}\\
(.53*.97) (2)\\
(.76*1) $A$\\
(.56*.89) $B$\\
(.56*.62) $C$\\ 
(.76*.52) $D$\\
(1*.62) $E$\\
(1*.9) $F$\\
(.53*.03) (3)\\
(.76*.46) $A$\\
(.56*.35) $B$\\
(.56*.08) $C$\\ 
(.74*0) $D$\\
(1*.08) $E$\\
(1*.35) $F$\\
\endSetLabels
\strut\AffixLabels{\includegraphics*[width=110mm]{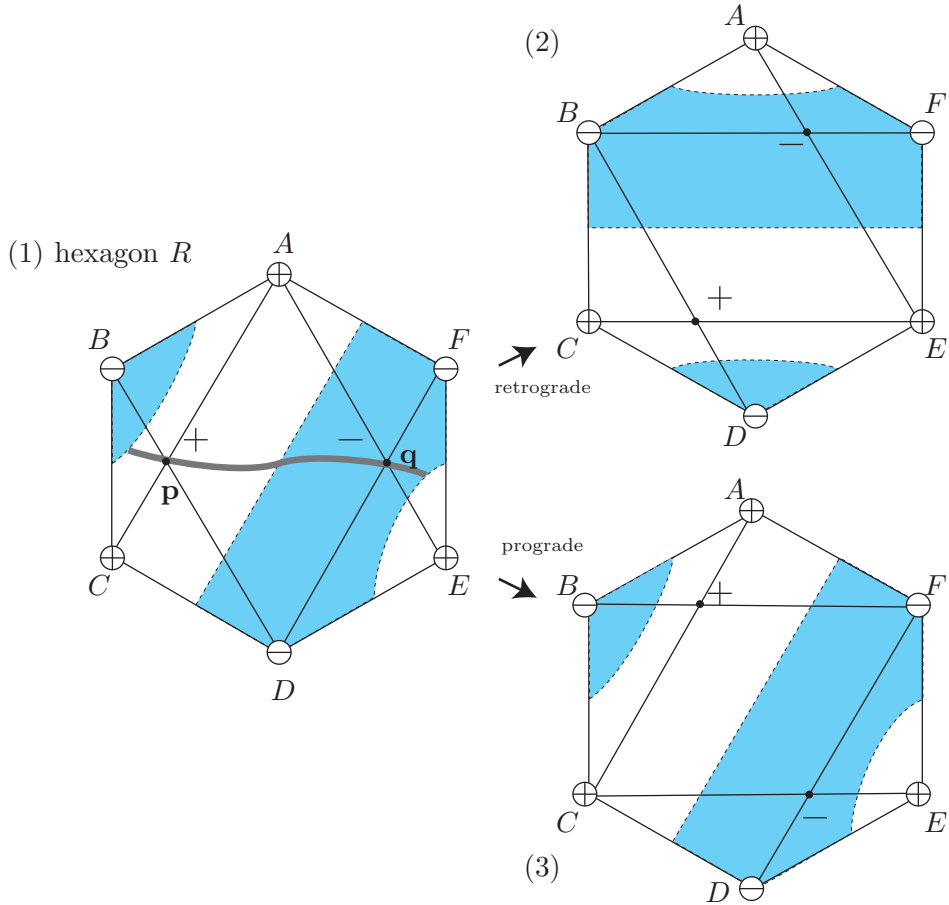}}
\caption{
(1) Original hexagon $R$. 
(2) Hexagon after retrograde bypass move.
(3) Hexagon after prograde bypass move.
Dashed arcs are dividing sets. Shaded regions are negative regions and unshaded regions are positive regions. }
\label{fig:bypass-hexagon}
\end{center}
\end{figure} 

Let  ${\bf p}, {\bf q}$ denote the two hyperbolic points of $R$. 
From now on we assume that $$\sgn({\bf p})=+1, \qquad \sgn({\bf q}) = -1.$$ 
(If $\sgn({\bf p})=-1$, $\sgn({\bf q}) = +1$ similar statements hold.) 
With this sign assumption there are two possible movie presentations realizing the open book foliation $\F(R)$ on $R$. See Figure~\ref{fig:bypass-movie}. We call them {\tt Type1} and {\tt Type2}. 
\begin{figure}[htbp]
\begin{center}
\SetLabels
(.05*.94) \small\tt{Type1}\\
(.05*.03) \small\tt{Type2}\\
(.11*.67) \small{$\bf p$}\\
(.54*.83) \small{$\bf q$}\\
(.54*.15) \small{$\bf q$}\\
(.07*.3) \small{$\bf p$}\\
\endSetLabels
\strut\AffixLabels{\includegraphics*[width=130mm]{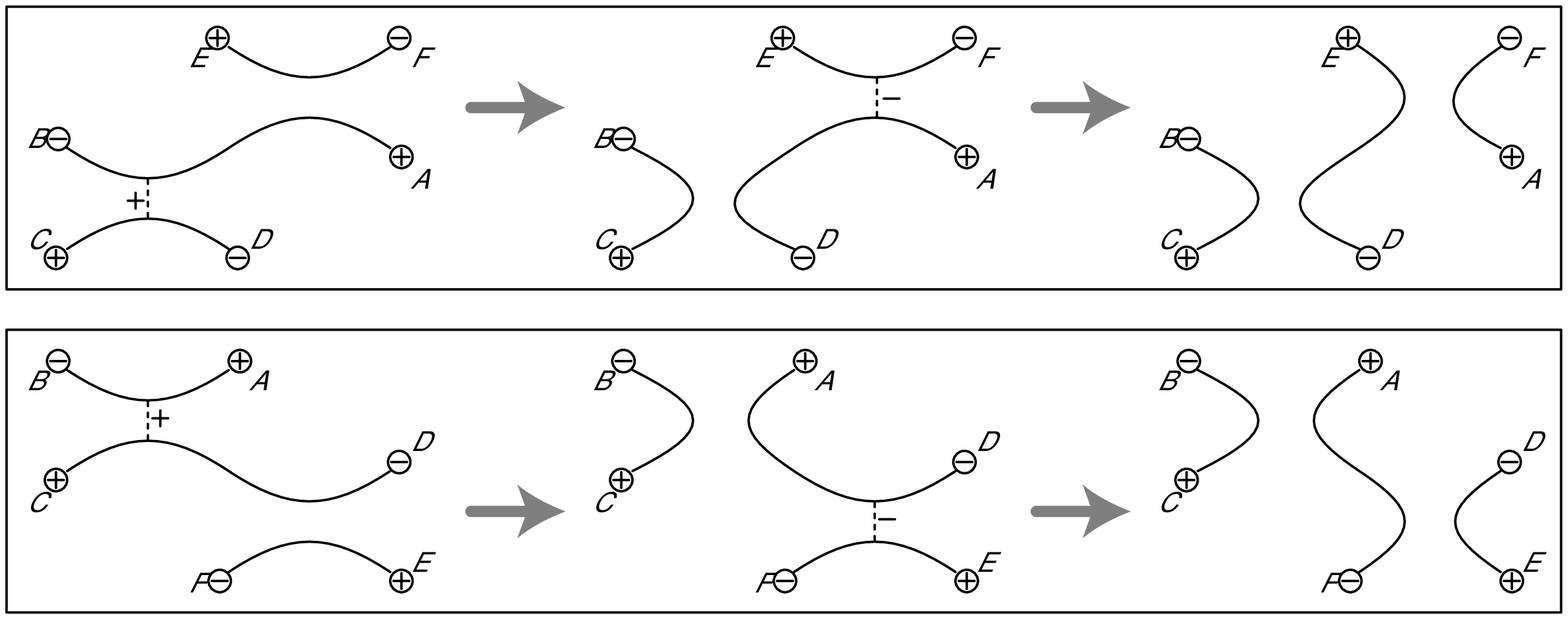}}
\caption{Movie presentations of {\tt Type1} and {\tt Type2} hexagon $R$.}\label{fig:bypass-movie}
\end{center}
\end{figure} 
\end{definition}

\begin{theorem}
If there exists a bypass rectangle $\D$ in $M \setminus F$ such that; 
\begin{enumerate}
\item
the union of arcs $\delta_1\cup \delta_2\cup \delta_4 \subset \partial \D$ is glued to the thick gray arc in Figure~\ref{fig:bypass-hexagon}-(1) that joins the dividing curves and contains {\bf p} and {\bf q}. 
\item
$p \in \D$ is identified with ${\bf p} \in R$
\item
$q \in \D$ is identified with ${\bf q} \in R$
\item 
$p$ and $q'$ live on the same page of the open book {\em(Figure~\ref{fig:bypass-movie-type1}-(1))}
\item
$p'$ and $q$ live on the same page of the open book {\em(Figure~\ref{fig:bypass-movie-type1}-(6))} 
\item $\sgn(\D)=\left\{
\begin{array}{cl}
+1 & \mbox{if $R$ is of {\tt Type1}}\\
-1 & \mbox{if $R$ is of {\tt Type2}}
\end{array}
\right.$
\end{enumerate} 
then by a local perturbation of $F$ supported on a neighborhood of $R \cup \D$, the open book foliation changes in the following ways: 
\begin{enumerate}
\item
If $R$ is of {\tt Type1}, $\F(R)$ changes as in the passage $(1) \rightarrow (2)$ of Figure~\ref{fig:bypass-hexagon} and the dividing set also changes. 
\item
If $R$ is of {\tt Type2}, $\F(R)$ changes as in the passage $(1) \rightarrow (3)$ of Figure~\ref{fig:bypass-hexagon} but the dividing set stays the same. 
\end{enumerate}
\end{theorem}

\begin{proof}
We study {\tt Type1} case carefully. 
Similar arguments work for {\tt Type2} case.

Figure~\ref{fig:bypass-movie-type1} shows a movie presentation of the bypass rectangle attached to a {\tt Type1} hexagon $R$ along $\delta_1, \delta_2, \delta_3$.
\begin{figure}[htbp]
\begin{center}
\SetLabels
(-.05*.12) (1)\\
(1.05*.12) (6)\\
(.15*.08) $p={\bf p}$\\
(.08*.02) $\bigstar$\\
(.35*.2) $q'$\\
(0*.02) $C$\\
(0*.21) $B$\\
(.17*.02) $D$\\
(.4*.02) $A$\\
(.23*.2) $E$\\
(.42*.2) $F$\\
(.63*.1) $p'$\\
(.96*.1) $q={\bf q}$\\
(.85*.12) $\bigstar$\\
(.79*.96) \tiny$(+)$\\
(.03*.5) $\delta_4$\\
(.08*.43) $\bigstar$\\
(-.05*.5) (2)\\
(1.05*.5) (5)\\
(.83*.55) $\bigstar$\\
(-.05*.88) (3)\\
(.12*.88) $\bigstar$\\
(1.05*.88) (4)\\
(.73*.9) $\bigstar$\\
(.14*.47) $\delta_1$\\
(.35*.58) $\delta_2$\\
(.2*.57) $\delta_3$\\
\endSetLabels
\strut\AffixLabels{\includegraphics*[width=120mm]{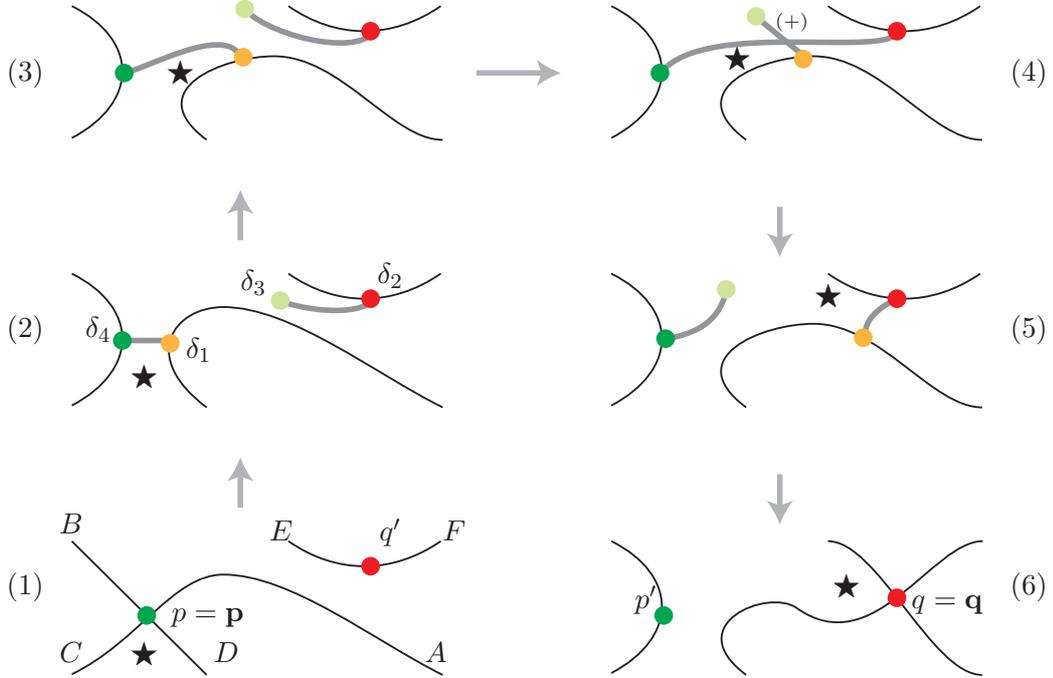}}
\caption{Movie presentation of a bypass rectangle $\D$ attached to a {\tt Type1} hexagon $R$. Orange dot $=\delta_1$, red $=\delta_2$, light green $=\delta_3$, green $=\delta_4$.}\label{fig:bypass-movie-type1}
\end{center}
\end{figure} 
Locally the bypass rectangle $\D$ and the hexagon $R$ are embedded as in the left sketch of Figure~\ref{fig:bypass-isotopy}. 
\begin{figure}[htbp]
\begin{center}
\SetLabels
(.25*-.1) \small{before perturbation $p={\bf p}, q={\bf q}$}\\
(.75*-.1) \small{after perturbation $p'={\bf p}, q'={\bf q}$}\\
(.1*.27) $p$\\
(.385*.67) $q$\\
(.07*.63) $p'$\\
(.42*.38) $q'$\\
(.08*.4) $\delta_4$\\
(.2*.4) $\delta_1$\\
(.2*.6) $\delta_3$\\
(.41*.53) $\delta_2$\\
(.25*.55) \tiny$\e=+1$\\
(.65*.27) $p$\\
(.9*.65) $q$\\
(.6*.6) $p'$\\
(.88*.3) $q'$\\
(.77*.55) \tiny$\e=+1$\\
\endSetLabels
\strut\AffixLabels{\includegraphics*[width=150mm]{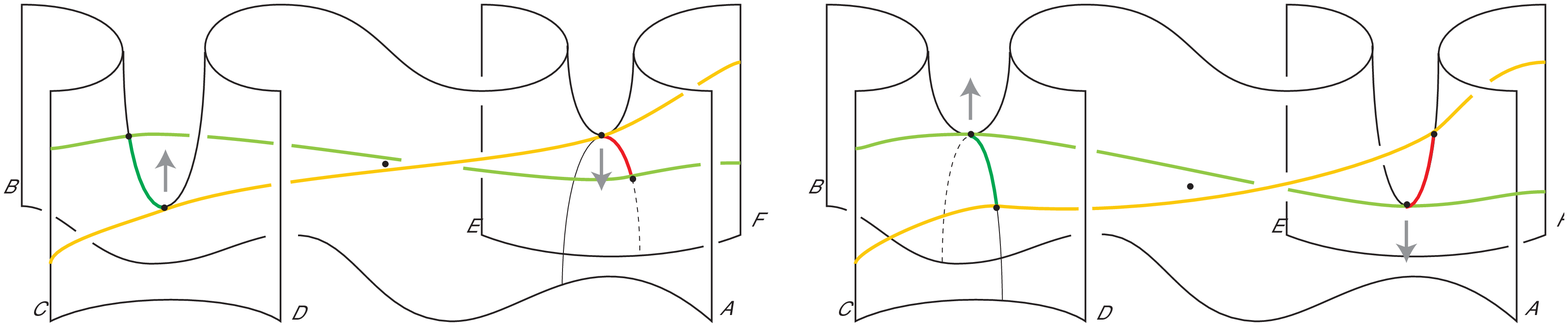}}
\caption{{\tt Type1} hexagon $R$ slid along the bypass rectangle $\D$.}\label{fig:bypass-isotopy}
\end{center}
\end{figure} 
The bypass plays a role of `stopper' that blocks other surfaces or braids (indicated by ``$\bigstar$'' in Figure~\ref{fig:bypass-movie-type1}) come from the region between $C$ and $D$, move through $p$ and $q$, and escape into the region between $A$ and $F$. 
Therefore we can slide the hexagon $R$ along the rectangle $\D$. 
This perturbation slides the hyperbolic point ${\bf p}$ along the arc $\delta_4$ from $p$ to $p'$. Similarly ${\bf q}$ is slid along $\delta_2$ from $q$ to $q'$.  
After the perturbation the arc $\delta_3$ sits on the new $R$ but $\delta_1$ no longer sits on the new $R$. 
See Figure~\ref{fig:bypass-move}.
\begin{figure}[htbp]
\begin{center}
\SetLabels
(.5*.5) retrograde bypass move\\
(.53*.79) $\delta_1$\\
(.56*.95) $\delta_2$\\
(.455*.89) $\delta_3$\\
(.425*.74) $\delta_4$\\
(.57*.09) $\delta_1$\\
(.59*.245) $\delta_2$\\
(.49*.2) $\delta_3$\\
(.46*.01) $\delta_4$\\
\endSetLabels
\strut\AffixLabels{\includegraphics*[width=145mm]{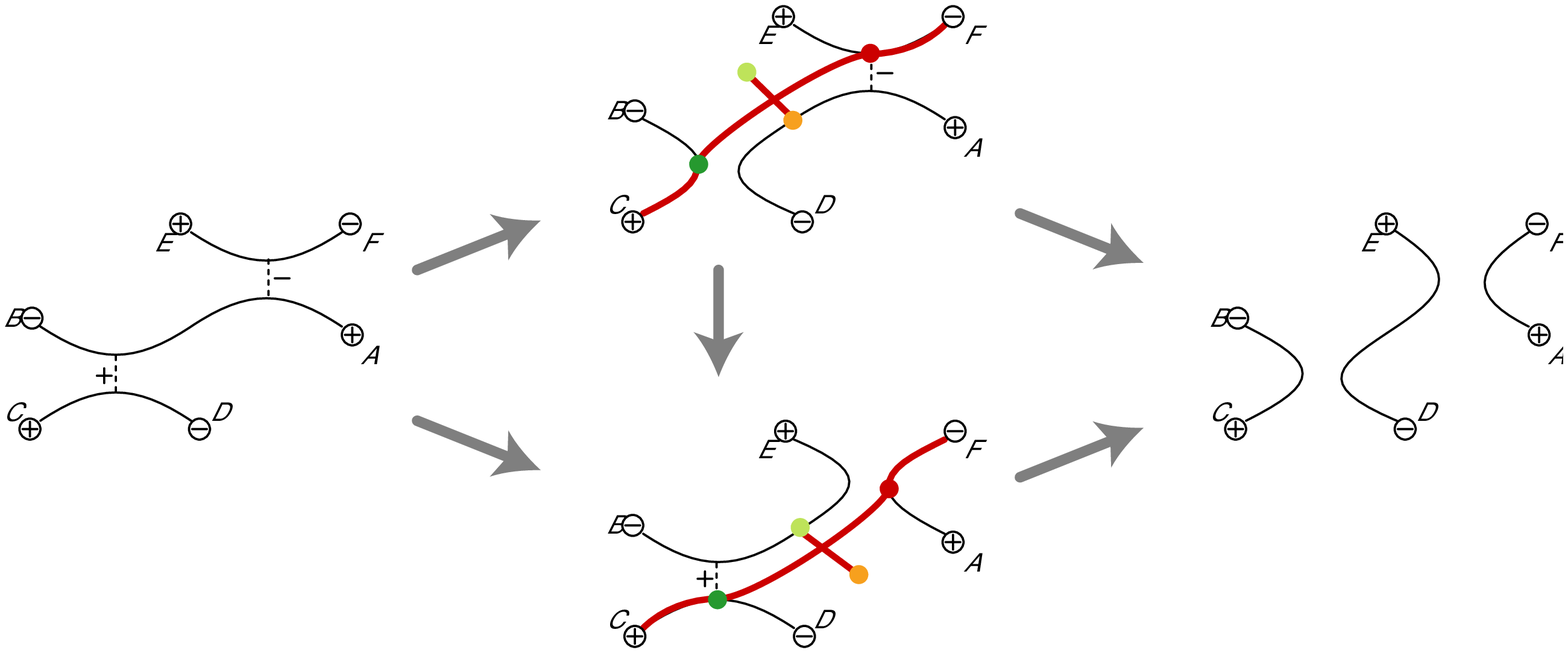}}
\caption{A bypass rectangle $\mathcal D$ (indicated by red arcs) before/after a retrograde bypass move applied on a {\tt Type1} hexagon. $\sgn(\mathcal D)=+1$. }\label{fig:bypass-move}
\SetLabels
(.5*.53) prograde bypass move\\
(.53*.82) $\delta_1$\\
(.56*.65) $\delta_2$\\
(.455*.92) $\delta_4$\\
(.45*.74) $\delta_3$\\
(.57*.08) $\delta_2$\\
(.57*.245) $\delta_1$\\
(.49*.19) $\delta_3$\\
(.46*.35) $\delta_4$\\
\endSetLabels
\strut\AffixLabels{\includegraphics*[width=145mm]{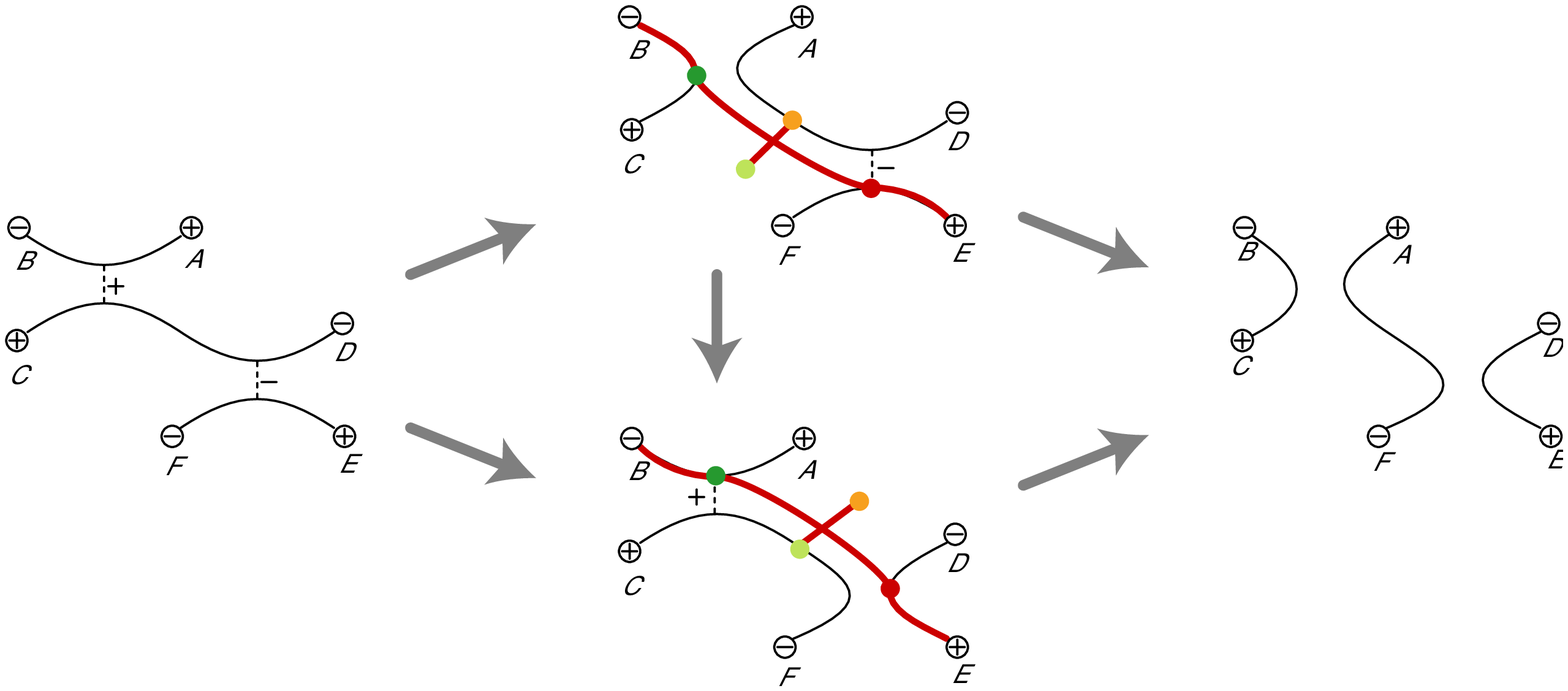}}
\caption{A bypass rectangle $\mathcal D$ before/after a prograde bypass move applied on a {\tt Type2} hexagon. $\sgn(\mathcal D) =-1$.}
\end{center}
\end{figure} 
\end{proof}

\begin{definition}[Retrograde/prograde bypass moves]
The above perturbation of {\tt Type1} hexagon region interchanges  the `heights' of the hyperbolic points {\bf p}, {\bf q}.  At one moment {\bf p}, {\bf q} have the same height. That is, {\bf p} and {\bf q} lay on the same page of the open book and they are joined by a singular leaf of the open book foliation. 
The singular leaf is oriented from {\bf q} to {\bf p}. 
Recall that $\sgn({\bf p})=+1$, $\sgn({\bf q}) = -1$.
Namely the singular leaf is oriented from a negative hyperbolic point to a positive hyperbolic point. 
Such a singular leaf is called a {\em retrograde saddle-saddle connection}.  
So we call the foliation change depicted in (1)$\rightarrow$(2) of Figure~\ref{fig:bypass-hexagon} {\em retrograde bypass move}.

On the other hand, for a corresponding perturbation of {\tt Type2} hexagon, the saddle-saddle connection is prograde, that is, the singular leaf is oriented from a positive hyperbolic point to a negative hyperbolic point. So we call the change in foliation depicted in (1)$\rightarrow$(3) of Figure~\ref{fig:bypass-hexagon} {\em prograde bypass move}. 
\end{definition}

We name the rectangle $\D$ {\em bypass} because our  retrograde bypass move and Honda's bypass attachment in convex surface theory yield exactly the same configuration change in dividing sets (compare Honda's \cite[Figure 6]{Honda} with our Figure~\ref{fig:bypass-hexagon}).

\begin{remark}
In his thesis \cite[p.123-124]{L}, LaFountain observes that Birman-Menasco's ``non-standard'' change of braid foliation that does change the graph $G_{++}$ is accomplished through a bypass.

In \cite{DP}, Dynnikov and Prasolov introduce a bypass for a rectangular diagram, a certain diagrammatic expression of (Legendrian) knots in the standard contact $S^{3}$. 
Their bypass can be turned into a Honda's bypass for the corresponding Legendrian link. 

Although in \cite{L, DP} techniques of braid foliations are extensively used we remark that our bypass and their bypass have difference. For example, we have two types of bypass moves, prograde and retrograde.  
\end{remark}

\section{Exchange moves}
\label{sec:exchangemove}

In this section we study {\em exchange moves} of open book foliations and closed braids.

First we recall the exchange move in braid foliation theory, which is one of the most fundamental operations on braid foliations and has numerous applications to study of knots and links in $S^3$ and transverse links in the standard contact $S^3$ \cite{bm2}.

An {\em exchange move} is a move of a closed braid in $S^{3}=M_{(D^2, id)}$ as depicted in Figure~\ref{fig:exchangemove}. 
It is a composition of a positive stabilization, braid isotopy and a positive destabilization. 
Suppose that braids $L, L'$ are related to each other by an exchange move. 
The conjugacy classes of $L$ and $L'$ are different in general but $L$ and $L'$ have clearly the same braid index and the same transverse link type \cite{bm2}.
\begin{figure}[htbp]
 \begin{center}
 \SetLabels
  (0.5*0.68) Exchange\\
  (0.5*0.58) move\\
  (0*.2) $L$\\
  (1*.2) $L'$\\
  (0.08*0.55)  \rotatebox{180}{\LARGE $A$}\\
    (0.66*0.55)  \rotatebox{180}{\LARGE $A$}\\
  (0.35*0.46) \LARGE $B$\\
  (0.92*0.46) \LARGE $B$\\
\endSetLabels
\strut\AffixLabels{\includegraphics*[scale=0.5, width=100mm]{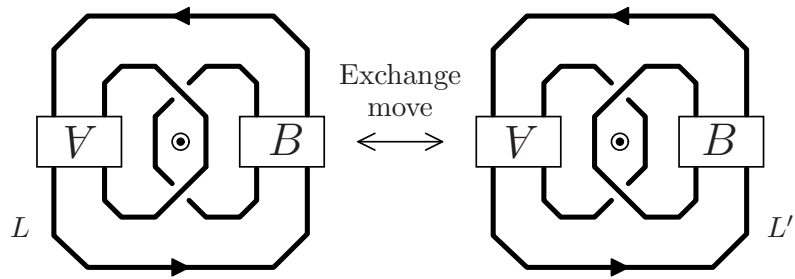}}
 \caption{Exchange move of a closed braid in $S^{3}$.}
 \label{fig:exchangemove}
  \end{center}
\end{figure}

Let $F$ be a Seifert surface of $L$ or an incompressible closed surface in $S^{3}\setminus L$. 
An exchange move of $L$ is related to an isotopy of the surface $F$. 
Consider a situation as depicted in Figure~\ref{fig:exchange2}-(1).  
\begin{figure}[htbp]
 \begin{center}
 \SetLabels
  (0.05*0.95) (1)\\
  (0.4*0.95) (2)\\
  (0.75*0.95) (3)\\
  (.05*.35) (4)\\
  (0.7*0.35) (5)\\
  (0.31*0.61) $L$\\
  (0.26*0.44) $F$\\
  (1*0.44) $F'$\\
  (-.05*.05) $\F(F)$\\ 
  (1.05*.05) $\F(F')$\\
\endSetLabels
\strut\AffixLabels{\includegraphics*[scale=0.5, width=100mm]{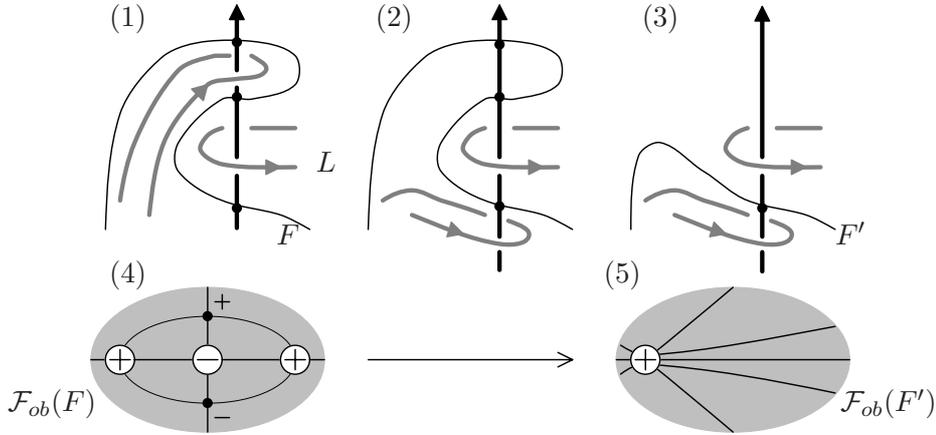}}
 \caption{Braid foliation before and after an exchange move of $L$.}
 \label{fig:exchange2}
  \end{center}
\end{figure}
We isotope $L$ as shown in passage $(1) \to (2)$ of Figure~\ref{fig:exchange2}. 
As a consequence inessential $b$-arcs appear in the braid foliation.  
Next we push down the surface to remove the inessential b-arcs (Figure~\ref{fig:exchange2}-(3)). 

The exchange move simplifies the braid foliation of $F$: It removes two elliptic points of opposite signs and two hyperbolic points of opposite signs in the (shaded) disc region of $F$, as described in Sketch (4)$\to$(5) of Figure~\ref{fig:exchange2} but it preserves the braid foliation on the rest of the surface.

The next theorem generalizes Birman-Menasco's exchange move. 

\begin{theorem}[Exchange moves in general open books]
\label{theorem:localex}
Let $L$ be a closed braid in $M_{(S,\phi)}$ and $F$ be a Seifert surface of $L$.
Assume that there exists a non-strongly-essential elliptic point $v \in \F(F)$ where exactly two regions $R_{1}$ and $R_{2}$ meet and satisfy the following: 
\begin{itemize}
\item 
$\sgn(R_1)=-\sgn(R_2)$.
\item
${\rm type} (R_1)= {\rm type} (R_2)=$ bb when $\sgn(v)=+1$.
\item
${\rm type} (R_1), {\rm type} (R_2) \in \{ ab,  bb\}$ when $\sgn(v)=-1$.
\end{itemize}
Then there exists an isotopy $\Phi_t: M\to M$ that takes $F=\Phi_0(F)$ to $F'=\Phi_1(F)$ and $L=\Phi_0(L)$ to $L' = \Phi_1(L)$ with the following properties: 
\begin{enumerate}
\item 
There exist discs $D\subset F$ and $D'\subset F'$ such that: 
\begin{enumerate}
\item 
$\F(F\setminus D)$ is topologically conjugate to $\F(F' \setminus D')$. 
\item 
$\F(D)$ has $\pm$ elliptic points and $\pm$ hyperbolic points as in Figure~\ref{fig:fol_exchange}-(1), but $\F(D')$ has no singularities as in Figure~\ref{fig:fol_exchange}-(3). 
\end{enumerate} 
\item 
$L$ and $L'$ have the same braid index w.r.t. the open book $(S, \phi)$ (but they may not be  isotopic in the complement of the binding). 
\item 
$L$ and $L'$ are transversely isotopic links in the contact structure $\xi_{(S, \phi)}$.
\end{enumerate}
\end{theorem}
\begin{figure}[htbp]
\begin{center}
\SetLabels
(0*.95) (1) $\F(F)$\\ (.6*.95) (3) $\F(F')$\\ (.13*.1) (2)\\
(0.06*.8)    $w_D$\\
(.07*.53) $D$\\ (.37*.8) $w$\\ (.93*.8) $w$\\ (.62*.29) $w$\\
(.21*.58) $\tiny \e$\\ (.21*.83) $\tiny -\e$\\
(.13*.8) $R_2$\\ (.13*.65) $R_1$\\
(.18*.7) $\tiny v$\\
(0.65*.7)   $D'$\\
\endSetLabels
\strut\AffixLabels{\includegraphics*[scale=0.5, width=110mm]{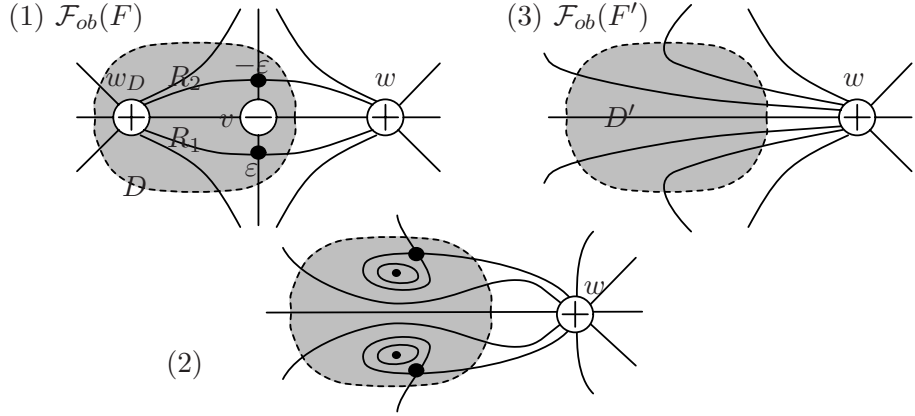}}\caption{An exchange move of open book foliation.} 
\label{fig:fol_exchange}
\end{center}
\end{figure}

\begin{definition}[Exchange moves]
\label{def:exchange}
{$ $}
\begin{enumerate}
\item[(i)]
We call the change $\F(F)\to\F(F')$ (in Thm~\ref{theorem:localex}) an {\em exchange move} of the open book foliation. 
\item[(ii)]
We call the braid move $L\to L'$ (in Thm~\ref{theorem:localex}) the {\em exchange move of $L$ subordinate to} the exchange move $\F(F)\to\F(F')$.
\end{enumerate}
\end{definition}

\begin{remark}
With a slight modification a similar statement as in Theorem~\ref{theorem:localex} holds when $F$ is a closed surface in $M\setminus L$. 
In fact in \S\ref{sec:braid theorem} we study a case where $F \simeq S^2$ and $\F(F)$ admits exchange moves. 
\end{remark}

\begin{remark}
Although in braid foliation theory $\F(F)$ in necessarily  essential, here we do not require essentiality of $\F(F)$. 
\end{remark}

\begin{proof}
We may assume that $\sgn(v)=-1$ and $\sgn(R_1)=-\sgn(R_2)=+1$. Similar arguments hold for other cases. 
Here is an outline of the proof:
\begin{description}
\item[(Step 1)] We define the surface $F''$ embedded in $M_{(S, \phi)}$ such that $\F(F'') = \F(F)$ topologically conjugate.
\item[(Step 2)] We find a continuous family of surfaces $\{F_t\}$ embedded in $M_{(S, \phi)}$ such that $F_0=F$ and $F_1=F''$. 
\item[(Step 3)] We construct $F'$ from $F''$ and verify (1) and (2). 
\item[(Step 4)] We verify (3). 
\end{description}

\noindent{\bf(Step 1)} 
For $i=1, 2$ let $h_{i}$ denote the hyperbolic point in $R_{i}$ and let $S_{t_{i}}$ be the singular fiber that contains $h_i$. 
For $t \neq t_1, t_2$, let $b_{t} \subset S_t$ be the b-arc of $\F(F)$ that ends at $v$.
Since $v$ is not strongly essential,
we may assume that $0<t_{1}<0.5 < t_{2}<1$, and $b_{t}$ is non-strongly essential for $t \in (t_1, t_2)$, thus $b_{t}$ and a binding component cobound a disc $\Delta_t$ in $S_{t}$.
%
%
Figure~\ref{movie-F} 
\begin{figure}[htbp]
\begin{center}
\SetLabels
(.07*1) $S_{t_1 - \epsilon}$\\ (.4*1) $S_{t_1}$\\ (.8*1) $S_{0.5}$\\ 
(.07*0) $S_{t_2 - \epsilon}$\\ (.4*0)   $S_{t_2}$\\ (.77*0) $S_0$\\
(.29*.62)   $v$\\ (.30*.82)   $w_D$\\ (-.02*.7) $w$\\ 
(0.13*.76)$\scriptscriptstyle{(+)}$\\
(.14*.23)$\scriptscriptstyle{(-)}$\\
(.07*.8)$A_1$\\ (.23*.72)$B_1$\\ (.13*.62)$D_1$\\ (.18*.89)$E_1$\\
(.43*.8)$A_2$\\ (.59*.72)$B_2$\\ (.49*.62) $D_2$\\ (.54*.89)$E_2$\\
(.79*.8)$A_3$\\ (.95*.72)$B_3$\\ (.862*.76)$C_3$\\
(.07*.26)$A_4$\\ (.23*.18)$B_4$\\ (.13*.13)$D_4$\\ (.155*.33)$E_4$\\ 
(.43*.26)$A_5$\\ (.59*.18)$B_5$\\ (.49*.07) $D_5$\\ (.54*.34)$E_5$\\ 
(.84*.24)$B_6$\\ (.85*.075) $D_6$\\ (.9*.34)$E_6$\\ (.9*.1) $b_0$\\
\endSetLabels
\strut\AffixLabels{\includegraphics*[width=140mm]{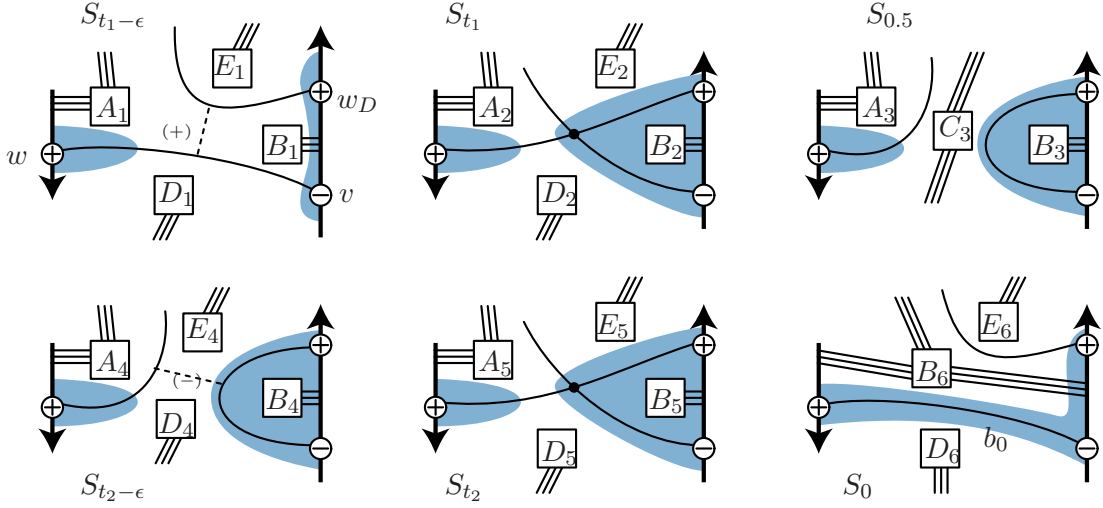}}
\caption{A movie presentation of $F$ near $X\cup T$.}
\label{movie-F}
\end{center}
\end{figure}
shows a movie presentation of $F$ in a neighborhood of $R_1\cup R_2$, where $\epsilon>0$ is a very small number, $w, w_D$ denote the positive elliptic points from which  $b_0, b_{0.5}$ start, and each box may contain  part of a-arcs, b-arcs, c-circles, and a singular leaf, or be empty. 
Triple parallel arcs represent some number (possibly zero) of  arcs, and the shaded regions indicate 
a neighborhood of $X \cup b_0$ where 
$$
X= \bigcup_{t_{1}< t < t_{2}} \Delta_{t} \cong D^{2} \times (t_{1},t_{2}). 
$$  

We define the surface $F''$ by replacing the part of $F$ described in Figure~\ref{movie-F} by the description in Figure~\ref{movie-F''}, 
\begin{figure}[htbp]
\begin{center}
\SetLabels
(.*1) $S_{t_1 - \epsilon}$\\ (.5*1) $S_{t_1}$\\ (.85*1) $S_{0.5}$\\ 
(.0*0) $S_{t_2 - \epsilon}$\\ (.4*0)   $S_{t_2}$\\ (.77*0) $S_0$\\
(.29*.62)   $v$\\ (.30*.82)   $w_D$\\ (-.02*.7) $w$\\ 
(0.13*.76)$\scriptscriptstyle{(+)}$\\
(.14*.23)$\scriptscriptstyle{(-)}$\\
(.07*.87)$A_1$\\ (.07*.77)\rotatebox[origin=c]{180}{$B_1$}\\ (.13*.62)$D_1$\\ (.18*.89)$E_1$\\
(.43*.87)$A_2$\\ (.43*.77)\rotatebox[origin=c]{180}{$B_2$}\\ (.49*.62) $D_2$\\ (.54*.89)$E_2$\\
(.79*.87)$A_3$\\ (.79*.77)\rotatebox[origin=c]{180}{$B_3$}\\ (.862*.76)$C_3$\\
(.07*.33)$A_4$\\ (.07*.24)\rotatebox[origin=c]{180}{$B_4$}\\ (.13*.13)$D_4$\\ (.15*.33)$E_4$\\ 
(.43*.33)$A_5$\\ (.43*.24)\rotatebox[origin=c]{180}{$B_5$}\\ (.49*.07) $D_5$\\ (.54*.34)$E_5$\\ 
(.815*.29)$B_6$\\ (.85*.075) $D_6$\\ (.9*.34)$E_6$\\ 
\endSetLabels
\strut\AffixLabels{\includegraphics*[width=140mm]{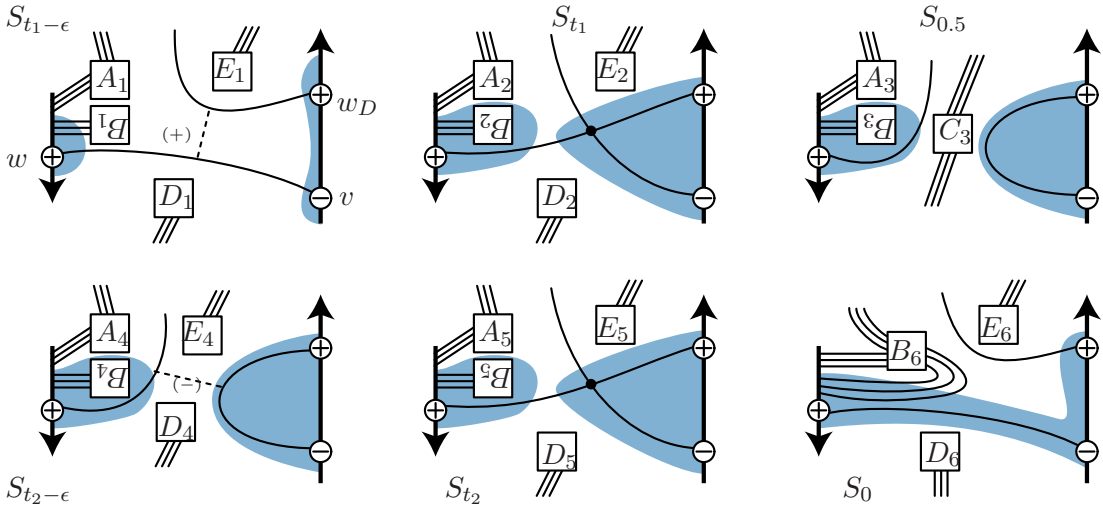}}
\caption{(Step 1) A movie presentation of $F''$ near $X\cup T$.}\label{movie-F''}
\end{center}
\end{figure}
which is obtained by moving the  boxes $B_1,\cdots,B_6 $ and their foot between $v$ and $w_D$ to the negative side of $w$.

By the construction, $F$ and $F''$ are homeomorphic and their open book foliations $\F(F)$ and $\F(F'')$ are topologically conjugate. 
If the open book foliation $\F(F)$ is essential then $B_i$ are non-empty, but in general all of $B_i$ can be empty and in that case $F''=F$. 
In any case the braids $L=\partial F$ and $\partial F''$ have the same braid index. 

\noindent{\bf(Step 2)}
By \cite[Lemma 4]{BM4} and \cite[Lemma 5]{BM5}, (see also \cite[Theorem 2.2, Figure 2.19]{bf}), there is an isotopy $\Phi'_t:M \to M$ that takes $F\cap X$ out of $X$ and move along $b_0$ down to the other tip as described in $(1)\to (6)$ of Figure~\ref{fig:local_braid}. 
We have $\Phi'_1(F)=F''$. 
%

\noindent{\bf(Step 3)}
By the construction of $F''$ the b-arcs $b_t$ of $\F(F'')$ for $t \in (t_{1},t_{2})$ are inessential (see Figure~\ref{movie-F''}). 
Push $F''$ along a disc $\Delta_t$ for some $t\in (t_1, t_2)$ as shown in Figure ~\ref{fig:removeb} to remove the inessential b-arcs and the elliptic points $w$ and $v$. Call the surface $F'''$. 
%
\begin{figure}[htbp]
\begin{center}
\SetLabels
(.18*.75) $w_D$\\
(.08*.25) $v$\\
(0.25*0.58) $\Delta_t$\\ 
(.3*.78) $F''$\\
(.85*.78) $F'''$\\
\endSetLabels
\strut\AffixLabels{\includegraphics*[scale=0.5, width=80mm]{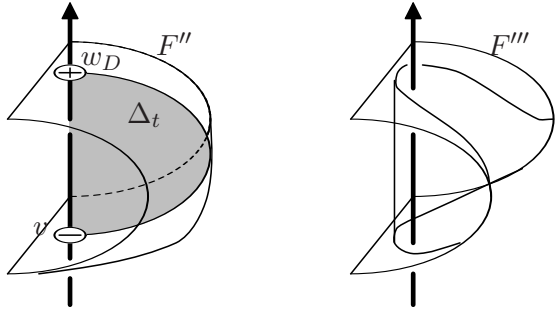}}
\caption{(Step 3) Push $F''$ along $\Delta_t$.}\label{fig:removeb}
\end{center}
\end{figure}

The surface $F'''$ does not admit an open book foliation as it has two local extrema, see Figure~\ref{fig:fol_exchange}-(2). 
Flatten the two pairs of local extremum and saddle tangency and call the resulting surface $F'$, whose open book foliation is depicted in Figure~\ref{fig:fol_exchange}-(3). 
This concludes the statement (1). 

If the boxes $B_1,\ldots,B_6$ are empty, the surface change $F \to F'$ is (the inverse of) what is called a {\em finger move} in \cite{ik1-1}.  

During the process $F'' \stackrel{\mbox{\tiny pushing}}{\longrightarrow} F''' \stackrel{\mbox{\tiny flatten}}{\longrightarrow} F'$ the boundary is fixed, so $L'=\partial F'$ and $\partial F''$ have the same braid index.  
With the observation at the end of Step 1, we verify the statement (2).

\noindent{\bf(Step 4)}
It remains to show the statement (3), that is, $L=\partial F$ and $L'=\partial F'$ are indeed transverse isotopic. 
So far we have three isotopies; $\Phi'_t$, pushing along $\Delta_{t}$, and the flattening. 
Denote the concatenation of the three by $\Phi_t:M \to M$, hence $\Phi_0(F)=F$ and $\Phi_0(F)=F'$. 
Note that $L_t=\Phi_t(L)$ may not be in a braid position relative to the open book for some $t\in (0,1)$.

To prove $L$ and $L'$ are transversely isotopic, we relate them by a sequence of positive (de)stabilizations and braid isotopy, all of which operations preserve transverse link types. 
We use an idea in Birman and Menasco's paper \cite[p.421]{bm1}: 
First we positively stabilize the part of $L$ that goes through $X$ along the $b$-arc $b_{0}$.
See Figure \ref{fig:local_braid}, 
\begin{figure}[htbp]
\begin{center}
\SetLabels
(0*.98) (1)\\
(.5*.98) (2)\\
(.88*.98) (3)\\
(0*.4) (4)\\
(.5*.4) (5)\\
(.88*.4) (6)\\
(0.3*.98)  stabilization\\
(0.5*0.5)  braid isotopy\\
(0.67*.98)  braid isotopy\\
(0.67*.4)  destabilization\\
(0.3*.4)  braid isotopy\\
(.12*.77) $b_0$\\ 
(.27*.87) $X$\\
(1.02*.3) $X$\\
(.22*.82) $v$\\
(.22*.72) $w$\\
(.22*.92) $w_D$\\
\endSetLabels
\strut\AffixLabels{\includegraphics*[scale=0.5, width=120mm]{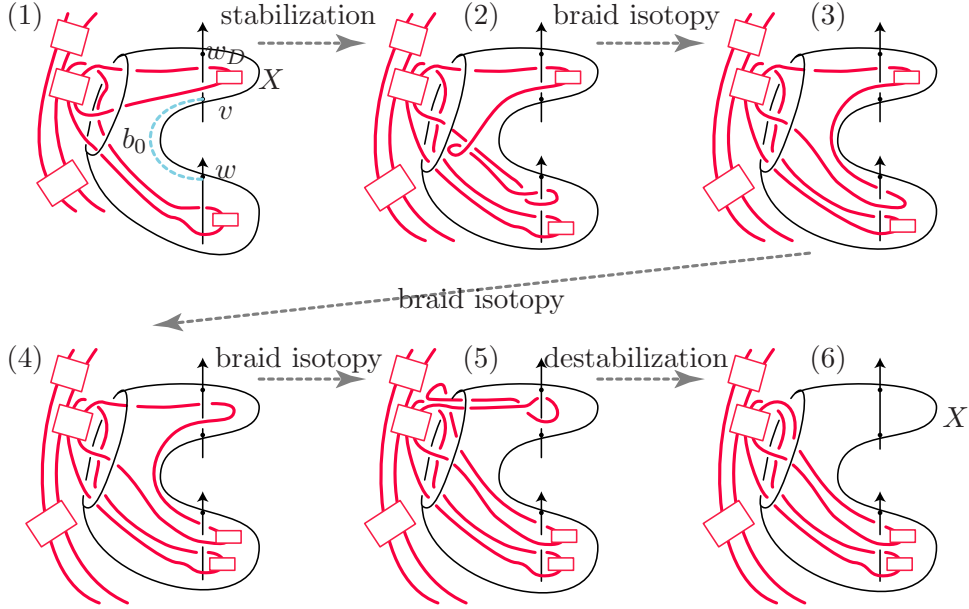}}
\caption{(Step 4) Realizing exchange move as transverse isotopy.}\label{fig:local_braid}
\end{center}
\end{figure}
where all the braid strands (red arcs) may be weighted and boxes contain braidings. 
After braid isotopy, we positively destabilize it so that the resulting braid $L'$ does not go through $X$. 
If $L$ does not go through $X$ then clearly $L=L'$. 
\end{proof}

Our exchange move is related to Giroux's elimination of a pair of elliptic and hyperbolic points of the same sign and connected by a singular leaf in a characteristic foliation.
In a neighborhood of $R_1 \cup R_2$ we may identify the open book foliation and the  characteristic foliation by the structural stability theorem in \cite{ik1-1}.
We see two elimination pairs in the shaded region of Figure~\ref{fig:fol_exchange}-(1). 
Applying Giroux elimination twice, we get a characteristic foliation topologically conjugate to Figure~\ref{fig:fol_exchange}-(3). 
Despite this fact, an exchange move and a Giroux elimination are different in the following sense: 
\begin{itemize}
\item 
A Giroux elimination can be achieved by a $C^{0}$-small perturbation that is supported on a small neighborhood of 
the singular leaf joining the elimination-pair, 
whereas the exchange move requires global isotopy (i.e., not $C^{0}$-small and not supported on a small neighborhood of $D$). 
Moreover the latter might change the braid isotopy class of $L=\partial F$ though it preserves the transverse link type.
\item 
One can apply a Giroux elimination without non-strongly essential condition on the elliptic point $v$, but for an  exchange move this assumption is necessary.  
\item
An exchange move on $\F(F)$ eliminates two pairs of elliptic and hyperbolic points at the same time. It is, in general, impossible to eliminate only one of the two pairs. 
But a Giroux elimination can apply to each pair separately. 
(In braid/open book foliation theory an operation called   destabilization of a closed braid eliminates one pair.) 
\end{itemize}

\section{Stabilization and open book foliations}
\label{sec:stabilization}

Let $(S, \phi)$ be an open book. Let $\alpha \subset S$ be a properly embedded arc in $S$. 
Let $S'$ denote the surface $S$ with an annulus $A$ plumbed along $\alpha$. 
Let $$\phi'_\pm:= D_\alpha^{\pm} \circ \tilde\phi \in {\rm Diff}^+(S', \partial S')$$
where $D_\alpha$ is the positive Dehn twist along a core circle of the attached annulus $A$, and $\tilde\phi:S'\to S'$ is an extension of $\phi:S\to S$ such that $\tilde\phi = \phi$ on $S$ and $\tilde\phi = id$ on $S' \setminus S.$ 
We call the new open book $(S', \phi'_{\pm})$ a {\em positive/negative  stabilization} of $(S, \phi)$ and the arc $\alpha$ a {\em stabilization arc}. 

It is known that (see Etnyre's survey \cite{e} for example)
$$M_{(S', \phi')} \simeq M_{(S, \phi)} \# M_{(A, D_\alpha)} \simeq M_{(S, \phi)} \# S^3 \simeq M_{(S, \phi)}.$$
so we may identify $M_{(S', \phi')}$ with $M_{(S,\phi)}$ by a homeomorphism  $\Theta: M_{(S,\phi)} \rightarrow M_{(S',\phi')}$ that preserves the pages,  $\Theta(S_t) = S'_{t}$ for $t \in [0,1)$.
 


In this section we study how the open book foliation $\F(F)$ of a surface $F \subset M_{(S, \phi)}$ changes under a stabilization of the open book $(S, \phi)$.
We start with a trivial case. 
Let $\alpha_t := \alpha \times \{t\} \subset S_t$. 

\begin{proposition}
If $F\subset M_{(S, \phi)}$ does {\em not} intersect $\alpha_0$ then there exists a surface $F' \subset M_{(S', \phi')}$ such that $F \simeq F'$ homeomorphic and that  $\F(F)\simeq\F(F')$ topologically conjugate. 
\end{proposition}

\begin{proof}
Let $\iota : S_t \hookrightarrow S_t'$ denote the natural inclusion map. 
We construct $F'$ so that $F' \cap S_t' = \iota(F \cap S_t)$ is satisfied for every $t\in[0,1)$. 
Then $F'$ does not intersect the arc $\iota(\alpha_0) \subset S_0'$.  Therefore
$$F' \cap S_0' = D_\alpha^{\pm1}( F' \cap S_0' ) =D_\alpha^{\pm1} \circ \iota \circ \phi( F \cap S_1 ) = D_\alpha^{\pm1} \circ \tilde \phi \circ \iota(F \cap S_1) = \phi'(F' \cap S_1'),$$
so we can identify the curves $F' \cap S_0'$ and $F' \cap S_1'$ by the monodromy $\phi'$ and obtain a surface $F' \subset M_{(S', \phi')}$. 
By the construction clearly $F \simeq F'$ and $\F(F)=\F(F')$
\end{proof}

Next we consider the case where $F$ intersects the stabilization arc $\alpha_0$ in the page $S_0$. 
Let $\overline\alpha_t \subset S_t$ be a collar neighborhood of $\alpha_t$. 
Assume that $F$ intersects $\overline\alpha_0$ in $m$ disjoint arcs $\beta_i \times\{0\}$;
$$F \cap \overline\alpha_0 = (\beta_1 \cup \cdots \cup \beta_m) \times \{0\}$$ 
where 
\begin{itemize}
\item
$\beta_i \subset S$ is an arc traversing the plumbed annulus $A$ (see Figure~\ref{fig:stabilization})
\item
the geometric intersection number $i(\beta_i, \alpha)=1$ and 
\item 
$\beta_i\times \{0\}$ is a sub-arc of some b-arc $b_i$ of the open book foliation $\F(F)$, possibly $b_i=b_j$ for some $i \neq j$. 
\end{itemize}

Then we construct two surfaces $F'$ and $F''$ in the stabilized open book $(S',\phi')$ that are homeomorphic to $F$.

\begin{proposition}\label{prop:F' and F''}
Suppose that $F\subset M_{(S, \phi)}$ intersects non-trivially the stabilization arc $\alpha_0$ in $m$ points. 
We further assume that $b_i\neq b_j$ for $i \neq j$, that is, every b-arc in $S_{0}$ intersects $\alpha_0$ in at most one point.
Then there exist surfaces $F'$ and $F''\subset M_{(S', \phi')}$ such that 
\begin{equation}\label{homeo}
F \simeq F' \simeq F''
\end{equation}
homeomorphic, and 
\begin{equation}\label{top_conjugate}
\F(F\setminus \mathcal D)\simeq\F(F' \setminus \mathcal D')\simeq\F(F'' \setminus \mathcal D'')
\end{equation}
topologically conjugate, where 
\begin{itemize}
\item
$\mathcal D \subset F$ is a disjoint union of $m$ bi-gons  foliated only by b-arcs,  
\item
$\mathcal D' \subset F'$ is a disjoint union of $m$ bi-gons each of which consists of two adjacent bb-tiles of opposite signs, 
\item
$\mathcal D'' \subset F''$ is exactly the same as $D'$ after exchanging the signs of the bb-tiles for each bi-gon. 
See Figure~\ref{fig:stabilization-tiles}.
\end{itemize}
\begin{figure}[htbp]
\begin{center}
\SetLabels
(0*1) $(1)$  $\mathcal D$\\
(-.02*.4) $b_i$\\
(.4*1) $(2)$  $\mathcal D'$\\
(.35*.05) $\beta_i' \times \{0\}$\\
(.65*.9) $\tau_i'\times\{0\}$\\
(.7*.8) $=\tau_i\times\{1\}$\\
(.7*.05) $\beta_i\times\{\e\}$\\
(.42*.5) $p_i'$\\
(.58*.5) $q_i'$\\
(.8*1) $(3)$  $\mathcal D''$\\
(.8*.5) $p_i''$\\
(.97*.5) $q_i''$\\
\endSetLabels
\strut\AffixLabels{\includegraphics*[width=110mm]{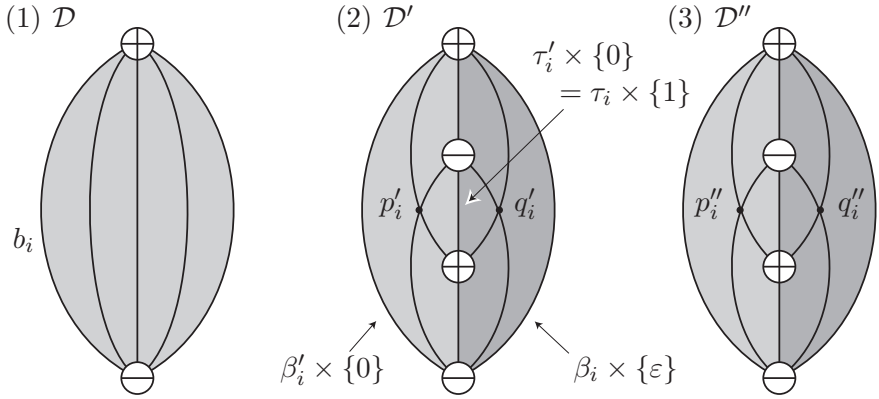}}
\caption{
(1) A bi-gon in $\mathcal D \subset F$.
(2) Two adjacent bb-tiles forming a bi-gon in $\mathcal D' \subset F'$.
(3) Two adjacent bb-tiles forming a bi-gon in $\mathcal D'' \subset F''$.
The hyperbolic points satisfy $\sgn(p'_i) = - \sgn(p_i'')= -\sgn(q_i') = \sgn(q_i'')$ for $i=1, \cdots, m$. 
}\label{fig:stabilization-tiles}
\end{center}
\end{figure}
\end{proposition}

\begin{remark}
If $b_{i_1}=b_{i_2}= \cdots = b_{i_k}$ for some $1\leq i_1<\cdots<i_k \leq m$ that is, if some b-arc in $S_{0}$ intersects the stabilization arc $\alpha_0$ in more than one point, after some modification of the descriptions of $\mathcal{D, D', D''}$, the same results (\ref{homeo}) and (\ref{top_conjugate}) still hold: 
For example, $|\mathcal D|=|\mathcal D'|=|\mathcal D''|$ is no longer $m$ but it becomes less than $m$.
Also Sketches (2) and (3) of Figure~\ref{fig:stabilization-tiles} become more complicated and each should contain $2(k+1)$ bb-tiles. 
\end{remark}

\begin{proof}
We prove Proposition~\ref{prop:F' and F''} only for the case $\phi'=\phi'_+$ (positive stabilization) since parallel argument holds for the case  $\phi'=\phi'_-$. 

We may assume that there exists $\e>0$ such that $\F(F)$ has no hyperbolic points in the family of pages $\{S_t\}_{0\leq t < 2\e}$ and that  
$$F \cap \overline\alpha_t = (\beta_1 \cup \cdots \cup \beta_m) \times \{t\} \quad \mbox{ for } 0 \leq t\leq \e.$$
We assume $\beta_1, \ldots, \beta_m$ are lined up from the left to the right as in Figurer~\ref{fig:stabilization}.  
Recall that $\beta_i \times \{0\}$ is a sub-arc of a b-arc $b_i \subset \F(F)$. The orientation of $b_i$ induces an orientation of $\beta_i$. 
Let $\tau_1, \ldots, \tau_m$ be essential arcs of $A$ lined up from the right to the left as in the left sketch of Figure~\ref{fig:stabilization}. 
We orient $\tau_i$ in the opposite direction to the orientation of $\beta_i$. (i.e., if $\beta_i$ is oriented ``upward'' then $\tau_i$ is oriented ``downward'' and vice versa.) 
\begin{figure}[htbp]
\begin{center}
\SetLabels
(.05*.9) $\tau_m$\\
(.1*.9) $\cdots$\\
(.15*.9) $\tau_1$\\
(.22*1) $\beta_1$\\
(.27*1) $\cdots$\\
(.32*1) $\beta_m$\\
(.35*.5) $\alpha$\\
(.4*.2) $A$\\
(.62*.9) $\tau_m'$\\
(.67*.9) $\cdots$\\
(.72*.9) $\tau_1'$\\
(.79*1) $\beta_1'$\\
(.84*1) $\cdots$\\
(.89*1) $\beta_m'$\\
\endSetLabels
\strut\AffixLabels{\includegraphics*[width=130mm]{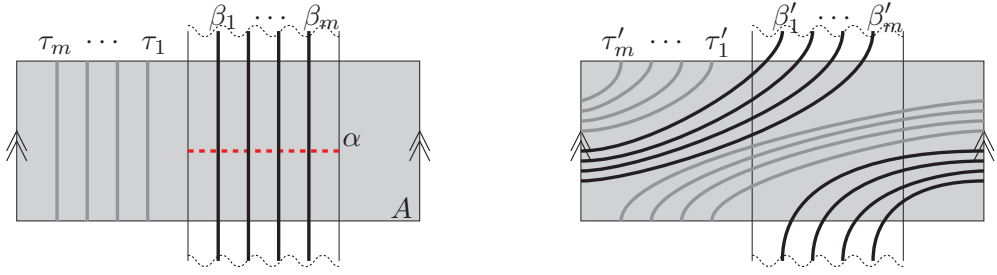}}
\caption{The arcs $\beta_i, \tau_i, \beta_i'= D_\alpha(\beta_i)$ and $\tau_i'= D_\alpha(\tau_i)$. 
The dashed line represents the stabilization arc $\alpha$. 
The shaded rectangle where the left and the right edges  are identified represents the plumbed annulus $A$.}\label{fig:stabilization}
\end{center}
\end{figure}

We construct $F'$ and $F''$ by defining intersection with the pages $S_t'$.

For $\e \leq t \leq 1$ let 
\begin{equation}\label{eq:e}
F' \cap S_t' = F'' \cap S_t' = \iota(F \cap S_t) \cup \bigcup_{i=1}^m (\tau_i \times \{t\}).
\end{equation}
Viewing the arc $\tau_i \times \{t\}$ as a b-arc of the open book foliation, the orientation of $\tau_i$ determines signs of the elliptic points of $\tau_i$.

When $t=0$ let 
$$F' \cap S_0' = F'' \cap S_0' = D_\alpha(\iota(F \cap S_0)) \cup \bigcup_{i=1}^m (\tau_i' \times \{0\}) = \phi'(F' \cap S_0'),$$ 
where $\tau_i' := D_\alpha(\tau_i)$ as in Figure~\ref{fig:stabilization}.

For $0\leq t \leq \e$ we define $F'$ and $F''$ by movie presentations. 
Let $\beta_i' := D_\alpha(\beta_i)$.
We make $\beta_i'$ and $\tau_i'$ come closer  
\begin{itemize}
\item
for $F'$ starting from $i=1$ to $m$ along the describing arcs in Figure~\ref{fig:stabilization-movie}-(1), 
\item
for $F''$ starting from $i=m$ to $1$ along the describing arcs in  Figure~\ref{fig:stabilization-movie}-(2).
\end{itemize}
\begin{figure}[htbp]
\begin{center}
\SetLabels
(0*0) (1) $F' \cap S_0'$\\
(.05*.28) $\tau_m'$\\
(.1*.28) $\cdots$\\
(.15*.28) $\tau_1'$\\
(.22*.32) $\beta_1'$\\
(.27*.32) $\cdots$\\
(.32*.32) $\beta_m'$\\
(1.02*0) (2) $F'' \cap S_0'$\\
(.6*.28) $\tau_m'$\\
(.65*.28) $\cdots$\\
(.7*.28) $\tau_1'$\\
(.77*.32) $\beta_1'$\\
(.82*.32) $\cdots$\\
(.87*.32) $\beta_m'$\\
(0*.38) (3) $F' \cap S_{\e/2}'$\\
(1.02*.38) (4) $F'' \cap S_{\e/2}'$\\
(.32*.9) $\tau_m$\\
(.37*.9) $\cdots$\\
(.41*.9) $\tau_1$\\
(.48*1) $\beta_1$\\
(.53*1) $\cdots$\\
(.58*1) $\beta_m$\\
(.25*1) (5) $F' \cap S_\e'= F'' \cap S_\e'$\\
\endSetLabels
\strut\AffixLabels{\includegraphics*[width=130mm]{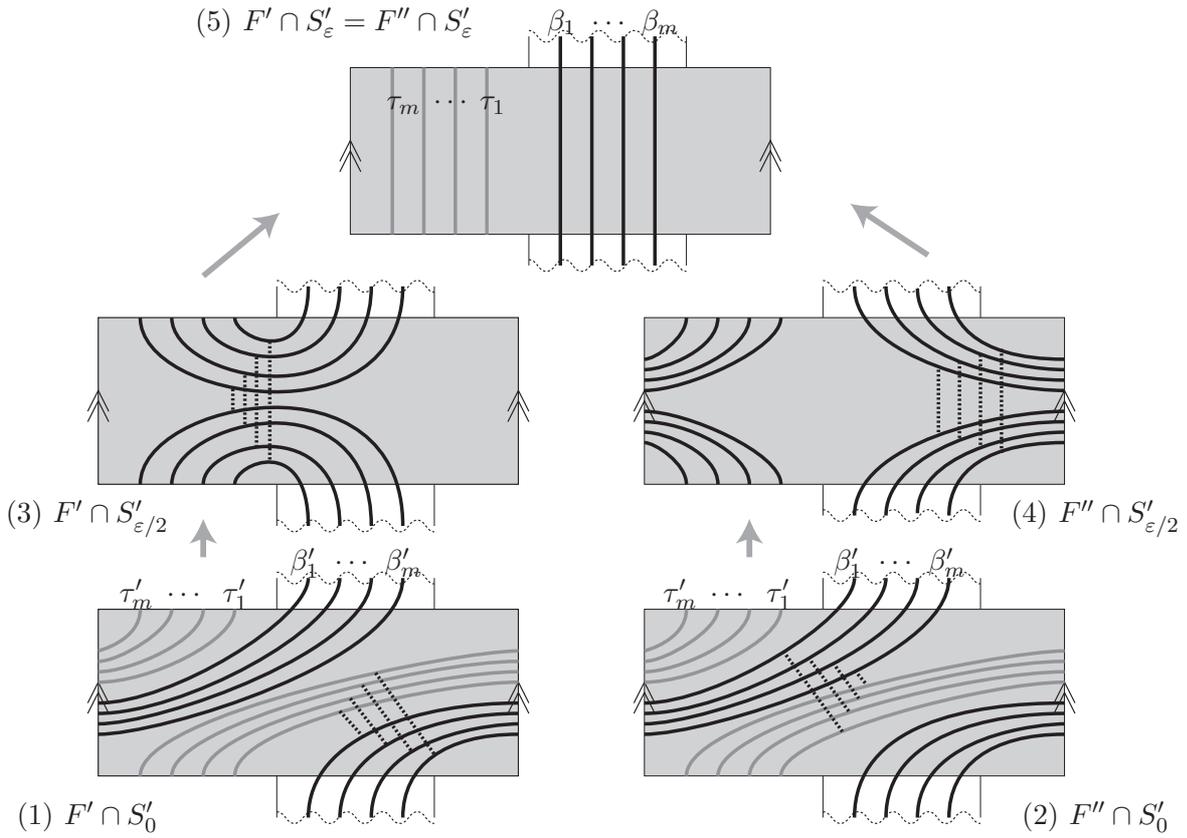}}
\caption{$(1)\to(3)\to(5)$: Movie presentation of $F'$.\\  
$(2)\to(4)\to(5)$: Movie presentation of $F''$. 
}\label{fig:stabilization-movie}
\end{center}
\end{figure}
Call the resulting saddle points $p'_i \in F'$ and $p_i'' \in F''$ respectively. 
Notice that we set the orientation of $\tau_i$ so that the hyperbolic points $p_i'$ and $p_i''$ have opposite signs. 

We further form hyperbolic points $q_m', \cdots, q_1'$ for $\F(F')$ and $q_1'', \cdots, q_m''$ for $\F(F'')$ by using the describing arcs as depicted in Figure~\ref{fig:stabilization-movie}-(3) and (4) respectively. 
On the level $t=\e$ the condition (\ref{eq:e}) is satisfied. 
We have
$$\sgn(p'_i) = - \sgn(p_i'')= -\sgn(q_i') = \sgn(q_i'')$$
and the bb-tiles of $\F(F')$ (resp. $\F(F'')$) containing $p_i'$ and $q_i'$ (resp. $p_i''$ and $q_i''$) are adjacent and form a bi-gon as depicted in Figure~\ref{fig:stabilization-tiles}. 
\end{proof}

We find similarity of the open book foliation $\F(\mathcal{D'})$ and Figure~ \ref{fig:fol_exchange}-(1). 
Pictorially the transition $\F(\mathcal D')\to \F(\mathcal D)$ in Figure~\ref{fig:stabilization-tiles}
caused by the destabilization $(S', \phi')\to(S, \phi)$ is the same as the transition $\F(F')\to\F(F)$ in Figure~ \ref{fig:fol_exchange} caused by an exchange move. 
Important differences are:  
\begin{itemize}
\item
For an exchange move the leaf corresponding to $\tau'_{i} \times \{0 \}$ must be boundary-parallel, whereas for a destabilization $\tau'_{i} \times \{0 \}$ is an essential arc.
\item
Under an exchange move the open book $(S, \phi)$ stays the same, but not under a destabilization. 
\end{itemize}

We have constructed two different surfaces $F'$ and $F''$ homeomorphic to $F$ in a stabilized open book. They are related to each other in the following way:

\begin{proposition}
\label{prop:F'=F''}
The surfaces $F', F'' \subset M_{(S', \phi')}$ constructed in the proof of Proposition~\ref{prop:F' and F''} are isotopic to each other. 
For example, they can be related to each other by exchange moves and bypass moves. 
See Figure~\ref{fig:stabilization-tile2}. 
$$F' \ {\rm sketch (1)}
 \stackrel{{\rm exchange}^{-1}}{\longrightarrow} 
{\rm sketch (2)} 
\stackrel{\rm bypass}{\longrightarrow} \
\stackrel{\rm bypass}{\longrightarrow}
{\rm sketch (3)} 
\stackrel{\rm exchange}{\longrightarrow} 
F'' \ {\rm sketch (4)}$$

\end{proposition} 

\begin{figure}[htbp]
\begin{center}
\SetLabels
(0*.95) (1) $\mathcal D'$\\
(.23*.02) exchange\\
(.23*-.04) move inverse\\
(.29*.95) (2) $\mathcal D_1$\\
(.5*.02) bypass\\
(.5*-.04) moves\\
(.56*.95) (3) $\mathcal D_2$\\
(.77*.02) exchange\\
(.77*-.04) move\\
(.83*.95) (4) $\mathcal D''$\\
\endSetLabels
\strut\AffixLabels{\includegraphics*[width=130mm]{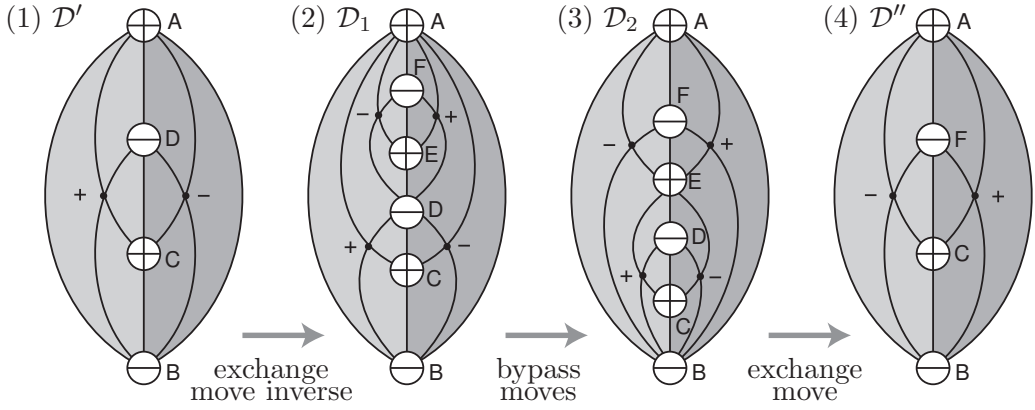}}
\caption{(1) bi-gon $\mathcal D' \subset F'$. 
(2) bi-gon $\mathcal D_1$. 
(3) bi-gon $\mathcal D_2$. 
(4) bi-gon $\mathcal D'' \subset F''$. 
}\label{fig:stabilization-tile2}
\end{center}
\end{figure}

\begin{proof}
For simplicity we assume that $m=1$, i.e., the number of bi-gon regions $|\mathcal D'|=|\mathcal D'|=1$ and we call the bi-gons $\mathcal D'$ and $\mathcal D''$, respectively, by abusing the notations.
(If $m>1$ each arc in Figure~\ref{fig:stabilization-tile2} is   replaced by parallel $m$ arcs and we apply similar constructions.)  
There are many ways to relate $\mathcal D'$ and $\mathcal D''$. In the following we present one of the ways.

Denote the elliptic points of $\mathcal D'$ by $A, B, C, D$ as in Sketch (1) of Figure~\ref{fig:stabilization-tile2} such that $\sgn(A)=\sgn(C)=-\sgn(B)=-\sgn(D)=+1$. 
We apply the inverse of an exchange move to $\mathcal D'$ to insert two adjacent bb-tiles between $A$ and $D$ as in Sketch (2), where $E$ and $F$ denote new positive and negative elliptic points, respectively. 
We call the resulting bi-gon of four bb-tiles $\mathcal D_1$.

Next we apply a retrograde bypass move to the left half of $\mathcal D_1$ and then apply a prograde bypass move to the right half of $\mathcal D_1$. 
Detailed movie presentation and bypass rectangles of the transition from $\mathcal D_1$ to $\mathcal D_2$ is depicted in Figure~\ref{fig:bypass-exchange}.

\begin{figure}[htbp]
\begin{center}
\SetLabels
(.4*0) (1')\\
(0*.2) (2')\\
(1*.2) (2'')\\
(.5*.31) retrograde\\
(.5*.27) bypass move\\
(.4*.4) (3')\\
(0*.6) (4')\\
(1*.6) (4'')\\
(.5*.7) prograde\\
(.5*.64) bypass move\\
(.4*1) (5')\\
\endSetLabels
\strut\AffixLabels{\includegraphics*[width=145mm]{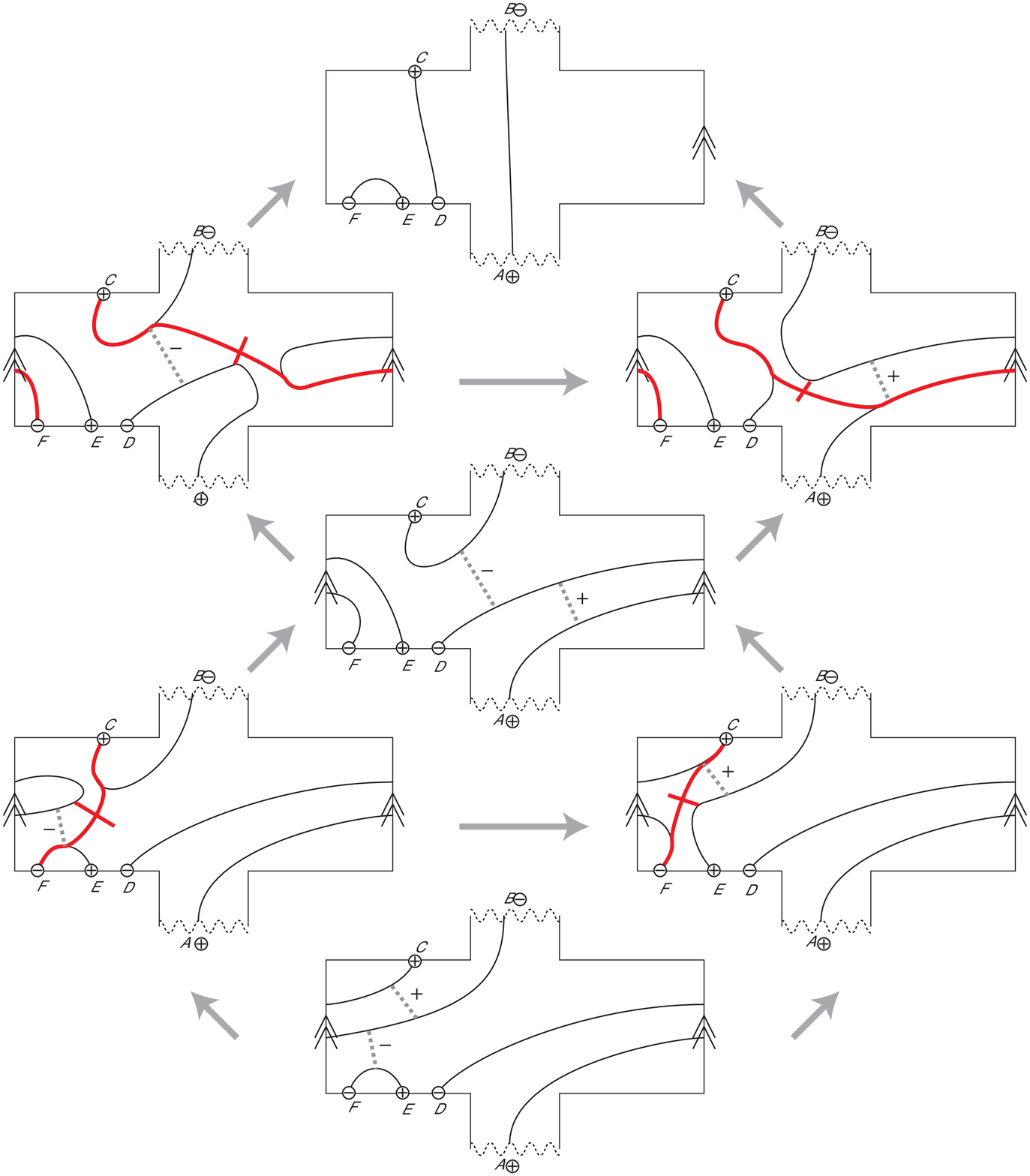}}
\caption{
(1' $\rightarrow$ 2' $\rightarrow$ 3' $\rightarrow$ 4' $\rightarrow$ 5') movie presentation of $\mathcal D_1$.\\
(1' $\rightarrow$ 2'' $\rightarrow$ 3' $\rightarrow$ 4'' $\rightarrow$ 5') movie presentation of $\mathcal D_2$. 
Thick arcs (red) represent bypasses.}
\label{fig:bypass-exchange}
\end{center}
\end{figure} 

Finally we get rid of two bb-tiles of $\mathcal D_2$ that share the elliptic points $D$ and $E$ by an exchange move and we obtain the bi-gon $\mathcal D''$. 
\end{proof}

\section{Split braid theorem and composite braid theorem}
\label{sec:braid theorem}

In this section we prove the split/composite braid theorem 
by using b-arc foliation change and exchange move.

\begin{definition}
Let $L$ be a link in a closed oriented $3$-manifold $M$. 
We say that $L$ is a {\em split link} if there exists a $2$-sphere that separates components of $L$. We call such a sphere a {\em separating sphere} for $L$.

Similarly, we say that $L$ is a {\em composite link} if there exists a $2$-sphere that intersects $L$ in exactly two points and decomposes $L$ as a connected sum of two non-trivial links. We call such a sphere a {\em decomposing sphere} for $L$.
\end{definition} 

The above notions of split/composite link are extended to those for closed braids relative to open books. 
(For braid foliations they are defined in \cite{BM4}.) 

\begin{definition}
Let $L \subset M_{(S, \phi)}$ be a closed braid w.r.t. $(S, \phi)$. 
We say that $L$ is a {\em split/composite closed braid} if there exists a separating/decomposing sphere $F$ for $L$ such that $\F(F)$ has exactly one positive elliptic point, one negative elliptic point and no hyperbolic points, namely $F$ intersects the binding in two points.
\end{definition}

Clearly a split/composite closed braid w.r.t. $(S, \phi)$ is a split/composite link in $M_{(S, \phi)}$, but the converse is not true in general.
This is because a separating/decomposing sphere might be embedded in quite complicated way relative to $(S, \phi)$. 
In fact, for the special case where $M_{(D^2, id)}\simeq S^{3}$ Birman and Menasco construct an example of split link  and its $4$-braid representative that cannot be isotopic to a split closed braid in the complement of the braid axis in \cite[p.116]{BM4}. 
Also in \cite{M} Morton find a $5$-braid representative of a composite link that is not conjugate to a composite $5$-braid. 

However, if we are allowed to use exchange moves the converse holds: 
In \cite{BM4} Birman and Menasco prove that any closed braid representative of a split/composite link in $S^{3}$ with the standard open book $(D^{2},id)$ can be modified to a split/composite braid by applying a sequence of exchange moves. As a corollary, they prove  the additivity of the minimum braid index of  knots and links in $\R^3$.

We extend the above result of Birman and Menasco to closed braids in general open books with additional assumptions. 
Let $C \subset \partial S$ be a boundary component of $S$. We denote by $c(\phi, C)$ the fractional Dehn twist coefficient of $\phi$ w.r.t. $C$, which is defined in \cite{hkm1} (cf.  \cite{GO}).

\begin{theorem}[Split/composite closed braid theorem]
\label{thm:split/composite}
Let $L$ be a closed braid representative of a split/composite link in $M_{(S,\phi)}$.
Let $F$ be a separating/decomposing sphere for $L$. 
Assume the following: 
\begin{enumerate}
\item
$\F(F)$ is essential and all of whose $b$-arcs are separating.
\item
If the binding component $C\subset \partial S$  intersects $F$ then $|c(\phi, C)| >1$.
\end{enumerate}
Then there exists a sequence of exchange moves of closed braids:
\[ L \rightarrow L_1 \rightarrow \cdots \rightarrow L_m \]
such that $L_m$ is a split/composite closed braid.
\end{theorem}

\begin{remark}
Before proceeding to a proof, we give several remarks on the assumptions and the statement of Theorem \ref{thm:split/composite}.

\begin{enumerate}
\item[(i)] The braid $L_{m}$ is split/composite and transversely isotopic to $L$. However, we do not assert that a separating/decomposing sphere $F_m$ for $L_m$ is isotopic to $F$. 
\item[(ii)] 
If $(S, \phi)$ has connected binding then by \cite[Thm 7.2]{ik2} the conditions (1), (2) imply that the sphere $F$ (hence $F_m$) bounds a $3$-ball in $M$. 


\item[(iii)] In braid foliation theory the condition (1) always holds but $c(id, \partial D^{2})=0$. 
To treat braid foliation case uniformly, it is often convenient to regard $c(id,\partial D^{2})=+\infty$. 
This is also true for other results like Corollaries 7.3, 7.4, and Theorem 8.3 in \cite{ik2}. 

\end{enumerate}
\end{remark}





\begin{example}\label{ex:movie-sphere}
In general without assuming the conditions (1) or (2), there may exist a closed braid representative $L$ of a split/composite link type 
whose separating/decomposing sphere does not admit a sequence of exchange moves that turns $L$ into a split/composite closed braid.

For example let $\phi= \id_S$ (i.e., $c(\phi, C)=0$) and $F$ be a splitting sphere of $L$ defined by the movie presentation in Figure~\ref{fig:movie-sphere}. 
The open book foliation $\F(F)$ consists of two bb-tiles. 
Since all the b-arcs are strongly essential $F$ does not admit exchange moves.

\end{example} 
\begin{figure}[htbp]
\begin{center}
\SetLabels
(.07*.45) $\star$\\
(.15*.53) $\star$\\
(.28*.45) $\star$\\
(.3*.53) $\star$\\
(.18*.57) $\diamond$\\
(.22*.45) $\diamond$\\
(.25*.4) $\diamond$\\
(.37*.8) $\star$\\
(.45*.88) $\star$\\
(.58*.8) $\star$\\
(.6*.88) $\star$\\
(.58*.95) $\diamond$\\
(.52*.75) $\diamond$\\
(.55*.75) $\diamond$\\
(.67*.45) $\star$\\
(.7*.53) $\star$\\
(.88*.45) $\star$\\
(.9*.53) $\star$\\
(.78*.55) $\diamond$\\
(.8*.45) $\diamond$\\
(.85*.4) $\diamond$\\
(.37*.1) $\star$\\
(.45*.18) $\star$\\
(.58*.1) $\star$\\
(.6*.18) $\star$\\
(.58*.25) $\diamond$\\
(.52*.05) $\diamond$\\
(.55*.05) $\diamond$\\
\endSetLabels
\strut\AffixLabels{\includegraphics*[width=135mm]{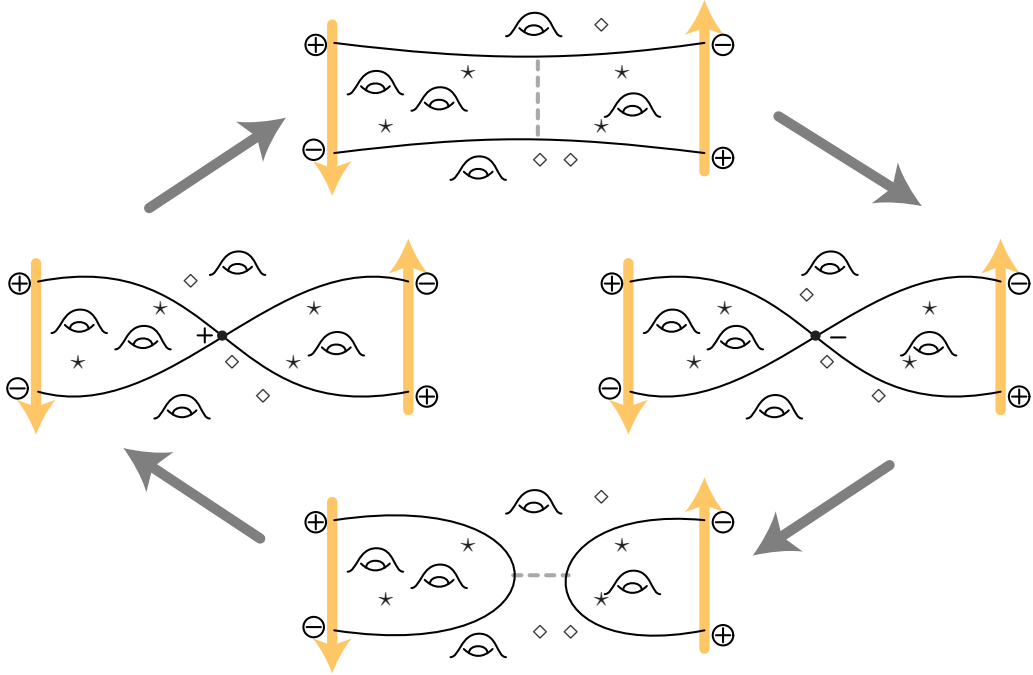}}
\caption{(Example~\ref{ex:movie-sphere}) 
A movie presentation of the separating sphere $F$ where $\star$ and $\diamond$ represent distinct  components of $L$ separated by $F$. 
}
\label{fig:movie-sphere}
\end{center}
\end{figure} 

We have three lemmas, where the  conditions (1) or (2) are not assumed.  
The first lemma is proven in \cite{ik2}. 

\begin{lemma}\label{lemma:estimate}\cite[Lemma 5.1]{ik2}
Let $(S, \phi)$ be a general open book and $F$ a closed, incompressible surface in $M_{(S, \phi)}$.
Let $v$ be a strongly essential elliptic point of $\F(F)$ that lies on a boundary component $C \subset \partial S$, and $P$ (resp. $N$) be the number of the positive (resp. negative) hyperbolic points that are connected to $v$ by a singular leaf. Then 
\[ 
\left\{
\begin{array}{ll}
-P \leq c(\phi,C) \leq N  & \text{ if } \sgn(v)=-1 \\
-N \leq c(\phi,C) \leq P  & \text{ if } \sgn(v)=+1.
\end{array}
\right.
\]
\end{lemma}

Here is the second lemma: 

\begin{lemma}
\label{lemma:sign}
Let $v$ be an elliptic point in the open book foliation $\F(F)$.
Assume that all the regions meeting at $v$ are bb-tiles, and that all the b-arcs that end on $v$ are separating.
Then there exist both positive and negative hyperbolic points connected to $v$ by a singular leaf. 
\end{lemma}

\begin{proof}
Let $h_1,\ldots,h_n$ be the hyperbolic points that is connected to $v$ by a singular leaf. 
We assume that $\sgn(v)=-1$ and $\sgn(h_i)=+1$ for all $i=1,\ldots,n$ (parallel arguments hold for other cases) and deduce a contradiction. 

Let $w_1,\ldots,w_n$ be the positive elliptic points that are connected to $v$ by a b-arc and ordered clockwise,  see Figure~\ref{fig:sepb}. 
Let $b_i$ be a b-arc in the page $S_{t_i}$ connecting $w_i$ and $v$, so $0<t_1<t_2< \cdots <t_n < 1$.  
\begin{figure}[htbp]
\begin{center}
\SetLabels
\endSetLabels
\strut\AffixLabels{\includegraphics*[scale=0.5, width=100mm]{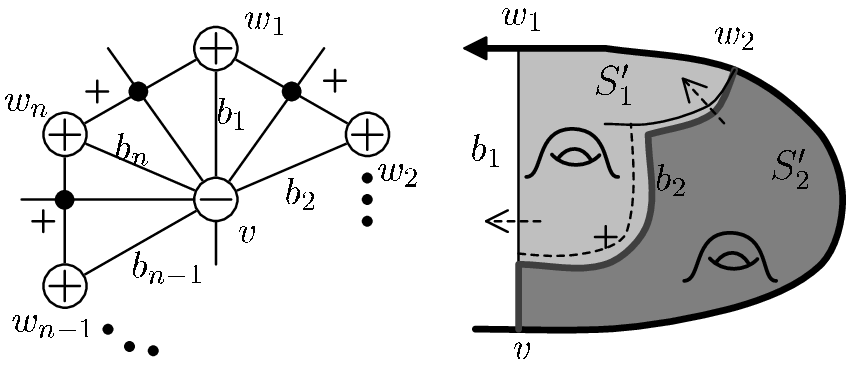}}
\caption{Lemma~\ref{lemma:sign}.}\label{fig:sepb}
\end{center}
\end{figure}
Since $b_i$ is separating the elliptic points $v$ and $w_i$ lie on the same binding component,  
i.e., $v$ and $w_1,\ldots,w_n$ lie on the same binding component. 
Let $S'_i \subset S_{t_i}$ be the subsurface that lies on the left side of $b_i$ as we walk from $w_i$ to $v$.
Since $\sgn(h_i)=+1$ by a standard argument (or the  argument as in the proof of Lemma 5.1 in \cite{ik2}) the describing arc of $h_i$ is contained in $S'_i$. 
Therefore $w_{i+1} \in S'_i$, hence $S'_i \supsetneq S'_{i+1}$ (see Figure \ref{fig:sepb}). 
In particular $w_{1} (=w_{n+1}) \in S'_n$. 
However, $S'_1 \supsetneq S'_2 \supsetneq \cdots \supsetneq S'_n$ and $w_1 \in S'_1\setminus S'_2$.  This is a contradiction. 
\end{proof}

Here is the third lemma.

\begin{lemma}
\label{lemma:degeneratebc}
Let $F \subset M_{(S, \phi)}$ be a closed incompressible surface in the complement of a closed braid $L$. 
We may assume that $\F(F)$ is essential by \cite[Thm 3.2]{ik2}.
Let $R$ be a degenerate bc-annulus in $\F(F)$, see Figure~\ref{fig:degeneratebc}.
Let $\mathcal{C} \subset S_{t_0}$ be the c-circle boundary of $R$ and $C \subset \partial S$ be a binding component that intersects $R$. 
If all the b-arcs in $\F(R)$ are separating then $\mathcal{C}$ is essential in $S_{t_0}$ and $|c(\phi,C)| \leq 1$.
\begin{figure}[htbp]
\begin{center}
\SetLabels
(.05*.88) $\mathcal C$\\
(.65*.88) $\mathcal C$\\
\endSetLabels
\strut\AffixLabels{\includegraphics*[width=90mm]{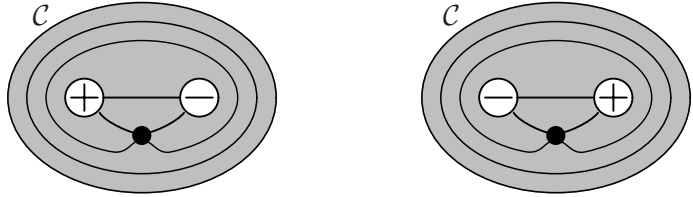}}
\caption{Degenerate bc-annuli $R$.}
\label{fig:degeneratebc}
\end{center}
\end{figure} 
\end{lemma}

\begin{proof}
Assume contrary that $\mathcal{C}$ bounds a disc $\Delta_{t_0} \subset S_{t_0}$, i.e., every c-circle of $\F(R)$ also bounds a disc $\Delta_t \subset S_t$. 
Since $F$ is incompressible in $M-L$, the disc $\Delta_{t_0}$ must be pierced by $L$ at least once. 
Since each $b$-arc $b_t \subset S_t \cap R$ is separating, 
$b_t$ cobounds a subsurface $S'_{t} \subset S_t$ that is disjoint from $\Delta_t$. 
Hence $R \cup\Delta_{t_0}$ bounds a compact region $M' \subset M$ which is the union of $S'_{t}$'s  and discs $\Delta_t$'s. 
Thus the algebraic intersection number of $L$ and $R \cup\Delta_{t_0}$ must be zero. 

On the other hand, since $L$ is a closed braid all the intersections of $L$ with $\Delta_{t_0}$ are positive. But  $L$ and $R$ never intersect, thus the algebraic intersection number of $L$ and $R \cup\Delta_{t_0}$ must be positive, which is a contradiction. 
This concludes that $\mathcal{C}$ is essential in $S_{t_0}$.

Moreover, if $\mathcal{C}$ is essential then all the b-arcs in $R$ are strongly essential (see \cite[Claim 6.8]{ik2}), hence by Lemma \ref{lemma:estimate} we have $|c(\phi,C)| \leq 1$.
\end{proof}

Now we are ready to prove Theorem~\ref{thm:split/composite}. 
Our proof is similar to Birman-Menasco's original one \cite{BM4}, but ours requires more careful and different approach, especially when we show nonexistence of c-circles (in {\bf Case II} below).  
More importantly, we need to be aware of the homotopical properties of b-arcs: essential, strongly essential or separating, 
since these properties are assumptions for b-arc foliation change and exchange move.

\begin{proof}[Proof of the split closed braid theorem] 

{$ $} 
\\

Let $F$ be a separating $2$-sphere 
with the essential open book foliation $\F(F)$. 
Let $e(F)$ be the number of elliptic points of $\F(F)$. We prove the theorem by induction on $e(F)$. 
We show that if $L$ is not a split closed braid (i.e., $e(F)>2$) then after applying a b-arc foliation change and an exchange move $e(F)$ decreases. Eventually we obtain $e(F)=2$, that is, $L$ is a split closed braid.
We study the following two cases:\\

{\bf Case I: $\F(F)$ contains no c-circle leaves}\\

In this case, the region decomposition of $F$ consists of bb-tiles only and it induces a cell decomposition of $F$. 
Let $V(i)$ ($i > 1$) be the number of $0$-cells (elliptic points) of valence $i$, $E$ the number of $1$-cells, and $R$ the number of $2$-cells (bb-tiles).
By the definition of bb-tiles, the valence of a 0-cell, $v$, is equal to the number of hyperbolic points that is connected to $v$ by a singular leaf. 
Notice that $V(1) = 0$ because existence of a 0-cell of valence $1$ implies existence of a degenerate bb-tile  which never exists. 
Since each $1$-cell is a common boundary of distinct two $2$-cells and each $2$-cell has distinct four $1$-cells on its boundary we have:
\begin{equation}
\label{eqn:ER}
2E=4R 
\end{equation}
Since the end points of each $1$-cell are distinct two $0$-cells we have:
\begin{equation}
\label{eqn:VE}
\sum_{i>1} i V(i) = 2E 
\end{equation}
The Euler characteristic of $F$ is:
\begin{equation}
\label{eqn:euler}
 \sum_{i>1} V(i) - E + R = \chi(F)=2
\end{equation}
From (\ref{eqn:ER}), (\ref{eqn:VE})  and (\ref{eqn:euler}), we get: 
\begin{equation}
\label{eqn:ecequality}
 \sum_{i> 1}(4-i)V(i) = 8
\end{equation}
The equality (\ref{eqn:ecequality}) implies: 
\begin{equation}
\label{eqn:last}
2V(2) + V(3) = 8 + \sum_{i \geq 4}(i-4)V(i) 
\end{equation} 
This shows that there exist vertices of valence $\leq 3$.

Assume that $v$ has valence 3. Let $h_1, h_2, h_3$ be the hyperbolic points that is connected to $v$ by a singular leaf. 
We may assume that $\sgn(h_1)=\sgn(h_2)$.  
Let $R_i$ denote the bb-tile that contains $h_i$. 
By Condition (1), the common b-arc of $R_1$ and $R_2$ is separating, so by Proposition~\ref{prop:sufficient-conditions} and Theorem~\ref{thm:folchange} we can apply a b-arc foliation change to $R_1 \cup R_2$, which lowers the valence of $v$ but preserves $e(F)$ and no c-circles are introduced.  

Hence we may assume that there exists a vertex of valence $=2$. 
Call it $v$. 
Let $C$ be the boundary component of $S$ on which $v$ lies.
By Condition (1) and Lemma \ref{lemma:sign} the two hyperbolic points around $v$ have opposite signs. 
If $v$ is strongly essential, Lemma \ref{lemma:estimate} implies $|c(\phi,C)| \leq 1$. 
This contradicts the condition (2), so $v$ is non-strongly essential. Hence by an exchange move on $\F(F)$ that involves an exchange move on $L$  we can remove $v$ and get a new splitting sphere $F'$ with $e(F') = e(F)-2$. 
We can repeat this procedure until we get $F$ with $e(F)=2$. 
\\

{\bf Case II: $\F(F)$ contains c-circle leaves}\\

In this case the region decomposition of $F$ contains  bc-annuli (and possibly cc-pants). 
Let $R$ be an {\em innermost} bc-annulus, here by `innermost' we mean that the c-circle boundary of $R$ bounds a disc $D$ such that $R\subset D \subset F$ and $D \setminus R$ contains no c-circles. Because $F$ is a sphere such $R$ necessarily exists and also a cc-pants cannot be innermost. 

If $R$ is degenerate (i.e., $D=R$) then by Lemma~ \ref{lemma:degeneratebc} we get a contradiction. 

Suppose that $R$ is non-degenerate. Then the region decomposition of $D_\circ:=D\setminus R$ consists only of bb-tiles. 
We can verify that the formula (\ref{eqn:last}) also holds for $\F(D_\circ)$. We apply a similar argument as in (Case I) to $D_\circ$ repeatedly until all the 0-cells in $\Int(D_\circ)$ disappear. Now the region $R$ is a degenerate bc-annulus, which is a contradiction. 

Therefore, under the conditions (1), (2) of the theorem, $\F(F)$ actually does not contain c-circles.  
\end{proof}

\begin{proof}[Proof of the composite closed braid theorem]
We prove the composite closed braid theorem in the same way as the split closed braid theorem (SCBT). 
The main difference between the two theorems is that a decomposing sphere $F$ has intersections  with $L$ but a splitting sphere does not.

By the same argument as in the embedded surface case \cite[Theorem 3.2]{ik2},
using Novikov-Roussarie-Thurston's general position argument \cite{t} we can put $F$ so that it admits an essential open book foliation.

If the region decomposition of $F$ consists only of bb-tiles the above equality (\ref{eqn:last}) holds. By the same argument as in (Case I) we may assume that $V(2)>0$. 
Except for the case $V(2)=4$ and $V(i)=0$ for $i=3, 4, \ldots$, we can move the intersection points $L \cap F$ by following the guideline in \cite[Lem 1]{BM4e} outside the region we attempt to apply an exchange move (the shaded region in Figure~\ref{fig:fol_exchange}-(1)). 
Then we apply an exchange move. The number $e(F)$ decreases by $2$ and no new c-circles are introduced.  
We repeat this procedure until $F$ satisfies $V(2)=4$ and $V(i)=0$ for $i=3, 4, \cdots$. 
This case is depicted in \cite[Fig 22]{BM4} by Birman and Menasco. Only the difference is the two b-arcs joining $p_2, p_3$ and $p_1, p_3$ in that figure may be strongly essential in our situation. By the argument in \cite[p.136]{BM4} our sphere $F$ admits one more exchange move and we obtain $e(F)=2$. 



We need to treat the case where $\F(F)$ contains c-circles.   
Let $R \subset F$ be an innermost bc-annuli. 
As in the proof of the SCBT, after exchange moves and b-arc foliation changes $R$ becomes a degenerate bc-annulus. 
By the proof of Lemma~ \ref{lemma:degeneratebc}, $R$ must have one non-empty intersection with $L$. 
We note that $\F(F)$ contains no cc-pants, because otherwise $F$ is capped off by (at least) three degenerate bc-annuli and all but two are not pierced by $L$ which contradicts Lemma~\ref{lemma:degeneratebc}.

Therefore up to isotopy we may consider that $F$ consists of two degenerate bc-annuli $R_1$ and $R_2$, each of which is pierced by $L$ (Fig~\ref{fig:degcase}-(1)).
We observe that all the b-arcs of $\F(F)$ are boundary-parallel: Because otherwise, by Lemma \ref{lemma:estimate} the condition (2) will be violated.  
All the c-circles of $\F(F)$ bound discs in their pages because otherwise, there must exist strongly essential b-arcs. 
Moreover each disc bounded by a c-circle is pierced by $L$ in one point. 
We replace $F$ with the degenerate bc-annulus $R_1$ capped off by the disc.
We perturb the disc to be foliated by concentric circles and has a local extremal point (Fig~\ref{fig:degcase}-(2)). 
Then flatten the extremal point paired with the hyperbolic point in $R_1$, this will turn $F$ into a desired decomposition sphere (Fig~\ref{fig:degcase}-(3)). 
During these operations the braid $L$ is fixed.  
\end{proof}


\begin{figure}[htbp]
\begin{center}
\SetLabels
(-.1*.95) (1)\\
(-.1*.6) (2)\\
(-.1*.25) (3)\\
(.55*.7) $L$\\
(.25*.94) $F$\\
(1*.95) $\F(F)$\\
(1.1*.89) $\leftarrow R_1$\\
(1.1*.75) $\leftarrow R_2$\\
(.03*.9) $R_1$\\
(.38*.83) $R_2$\\
\endSetLabels
\strut\AffixLabels{\includegraphics*[width=90mm]{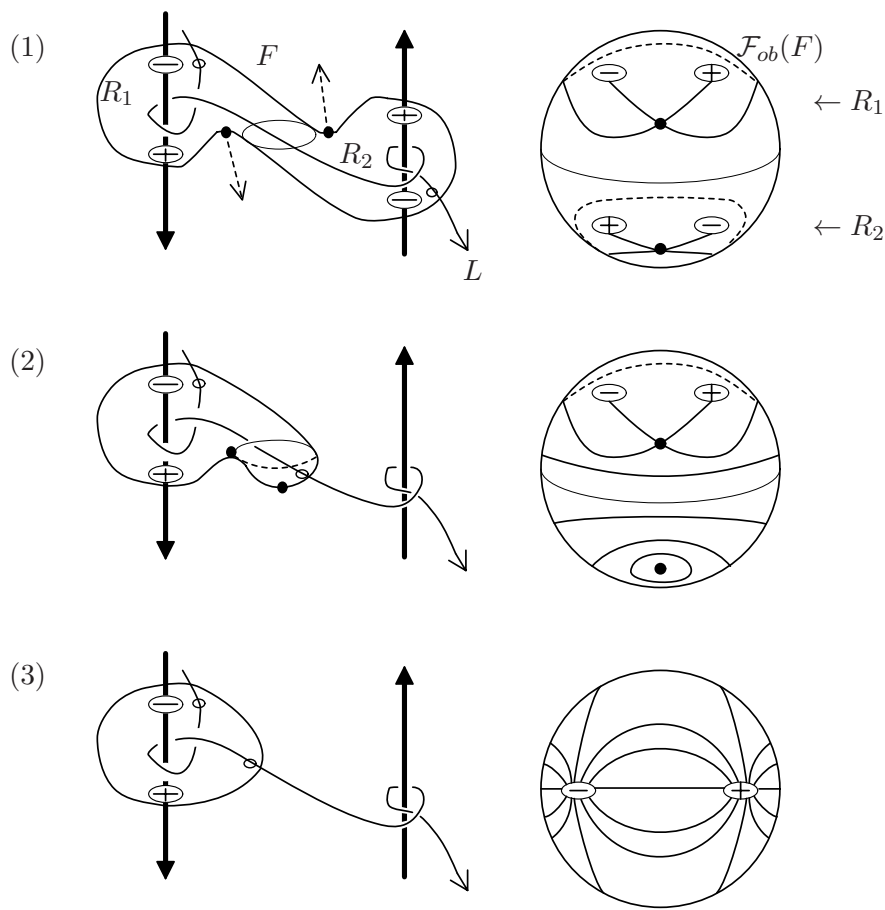}}
\caption{Special case: a decomposing sphere consisting of two degenerate bc-annuli.}
\label{fig:degcase}
\end{center}
\end{figure}

\section*{acknowledgement}
The authors thank Bill Menasco for constructive conversations on b-arc foliation change and Doug LaFountain for turning their attention to bypass moves.

\end{document}